\documentclass[12pt]{amsart}
\usepackage[pdftex]{graphicx,color}
\usepackage{amssymb,latexsym,bm,color}
\usepackage[dvips,centering,includehead,width=15.1cm,height=22.7cm]{geometry}
\usepackage{amsmath}
\usepackage{epsfig}
\usepackage{amssymb}
\input{xy} \xyoption{all}
\usepackage{color}

\definecolor{darkgreen}{rgb}{0.0, 0.7, 0.0}

\numberwithin{equation}{section}

\newtheorem{theorem}{Theorem}[section]

\newtheorem{corollary}[theorem]{Corollary}
\newtheorem{proposition}[theorem]{Proposition}

\theoremstyle{definition}
\newtheorem{definition}[theorem]{Definition}
\newtheorem{example}[theorem]{Example}
\newtheorem{remark}[theorem]{Remark}

\newtheorem{notation}[theorem]{Notation}

\newcommand{\thmref}[1]{Theorem~\ref{#1}}
\newcommand{\secref}[1]{\S\ref{#1}}

\newcommand{\corref}[1]{Corollary~\ref{#1}}
\newcommand{\subsecref}[1]{\S\ref{#1}}

\newcommand{\defref}[1]{Definition~\ref{#1}}
\newcommand{\remref}[1]{Remark~\ref{#1}}

\newcommand{\exmpref}[1]{Example~\ref{#1}}

\newcommand{\Coeff}{\mathop{\text{\rm Coeff}}}
\newcommand{\<}{\langle}
\renewcommand{\>}{\rangle}

\renewcommand{\H}{{\mathcal H}}

\newcommand{\D}{{\mathcal D}}

\renewcommand{\O}{{\mathcal O}}

\DeclareMathOperator {\Area}{Area}
\DeclareMathOperator {\mult}{mult}

\DeclareMathOperator {\dive}{div}

\renewcommand{\aa}{ {\bf a} }

\newcommand{\RR}{{\mathbb R}}
\newcommand{\ZZ}{{\mathbb Z}}
\newcommand{\Z}{{\mathbb Z}}

\newcommand{\Q}{{\mathbb Q}}

\newcommand{\PP}{{\mathbb P}}

\newcommand{\oooo}{\multiput(0,0)(10,0){4}{\circle{2}}}

\newcommand{\Eeee}{\put(1,0){\line(1,0){8}}}
\newcommand{\eEee}{\put(11,0){\line(1,0){8}}}
\newcommand{\eeEe}{\put(21,0){\line(1,0){8}}}

\begin{document}
\title[Fock spaces and refined Severi degrees]{Fock spaces and refined Severi degrees}

\author{Florian Block \and Lothar G\"ottsche}
\address{Florian Block, Department of Mathematics, University of
  California, Berkeley, Berkeley, USA}
\email{block@math.berkeley.edu}
\address{Lothar G\"ottsche, International Centre for Theoretical Physics, Strada Costiera 11, 34151 Trieste, Italy}
\email{gottsche@ictp.it}
\thanks{The first author was supported by a Feodor Lynen-Fellowship of the Alexander von Humboldt-Foundation.}

\begin{abstract}
A convex lattice polygon $\Delta$ determines a pair $(S,L)$ of a toric surface together with 
an ample toric line bundle on $S$. The Severi degree $N^{\Delta,\delta}$ is the number of $\delta$-nodal curves in the complete linear system $|L|$ passing through $\dim|L|-\delta$ general points.
Cooper and Pandharipande showed that in the case of  $\PP^1\times
\PP^1$ the Severi degrees can be computed as the matrix elements of an
operator on a Fock space.
In this note we want to generalize and extend this result in two ways. First we show that it holds more generally for $\Delta$ a so called 
$h$-transverse lattice polygon. This includes the case of $\PP^2$ and rational ruled surfaces, but also many other, also singular, surfaces. Using a deformed version of the Heisenberg algebra, we extend the result to the refined Severi degrees defined and studied by G\"ottsche and Shende  and by  Block and G\"ottsche. 
For $\Delta$ an $h$-transverse lattice polygon, one can, following
Brugall\'e and Mikhalkin,
 replace the count of 
tropical curves by a count of marked floor diagrams, which are slightly simpler combinatorial objects.  We show that these floor diagrams are the Feynman diagrams of certain operators on a Fock space, proving the result. 
\end{abstract}
\maketitle

\section{Introduction}
\label{sec:intro}

A $\delta$-nodal curve is a reduced (not necessarily irreducible) curve with $\delta$ simple nodes as only singularities. 
The \emph{Severi degrees} $N^{d,\delta}$ are the degrees of the Severi varieties parametrizing $\delta$-nodal plane curves of degree $d$. Equivalently, $N^{d, \delta}$ is the number of $\delta$-nodal 
plane curves of degree $d$ through $\tfrac{(d+3)d}{2} - \delta$ generic points in
the complex projective plane $\PP^2$.
More generally, for $(S,L)$ a pair of a surface and line bundle on $S$, the Severi degree $N^{(S,L),\delta}$ is the number of $\delta$-nodal curves in the complete linear system $|L|$ through $\dim|L|-\delta$ general points in $S$.
In this paper we will deal with the case that $S$ is a toric surface and $L$ a toric line bundle. And slightly contrary to the above we denote $N^{(S,L),\delta}(y)$, the number of 
of cogenus $\delta$ curves in $|L|$ passing though $dim |L| - \delta$ general points in $S$, which do
not contain a toric boundary divisor as a component, and analogously for the Welschinger invariants below.

The Severi degrees $N^{d,\delta}$ can be computed by the well-known Caporaso-Harris recursion formula
\cite{CH98}. Similar recursion formulas exist for other rational surfaces, in particular for rational ruled surfaces
\cite{Va00}. The Welschinger invariants $W^{(S,L),\delta}(P)$ are the analogues of the Severi degrees and the Gromov-Witten invariants in real algebraic geometry. Under suitable assumptions they are a count of  real algebraic curves on a real algebraic surface $S$, through a real configuration $P$ of $\dim|L|-\delta$ general points. Differently from the Severi degrees, the $W^{(S,L),\delta}(P)$ in general depend on the point configuration $P$.  
The Severi degrees and the Welschinger invariants of toric surfaces
can be computed via tropical geometry.

In \cite{GS12}  and \cite{BG14} 
\emph{refined Severi degrees} $N^{d,\delta}(y)$ and
$N^{(S,L),\delta}(y)$ are defined for toric surfaces, first for
$\PP^2$  on rational ruled surfaces via a Caporaso-Harris type
recursion 
 and then in general via tropical geometry. 
By definition 
$N^{(S,L),\delta}(1)=N^{(S,L),\delta}$ and
$N^{(S,L),\delta}(-1)=W^{(S,L),\delta}(P)$, for some $P$.

In the recent paper \cite{CoPa12} the Severi degrees of
$\PP^1\times\PP^1$ are expressed as matrix elements of operators on a
Fock space, and applications to other rational surfaces are
given. This note grew out of an attempt to reprove their results in
terms of tropical geometry, and use this approach to extend them to
the refined Severi degrees defined in \cite{GS12}, \cite{BG14}. 
This is done by using combinatorial gadgets called marked floor diagrams, which have been used in tropical enumerative geometry for some time. 
It turns out that this leads  to a generalization of the results of \cite{CoPa12}.
There is a  dictionary relating expressions in the Heisenberg algebra operators to marked floor diagrams, and this can be used to give formulas for the refined Severi degrees as matrix elements of operators on a Fock space, whenever the refined Severi degrees can be expressed in terms of marked floor diagrams. 
This includes the refined Severi degrees of $\PP^2$, rational ruled
surfaces, weighted projective spaces,
and more generally $h$-transverse lattice polygons as studied in
\cite{AB10}, \cite{BM2}.
Here in the introduction we state our result for $\PP^2$ and rational ruled surfaces $\Sigma_m$.
 We denote by $H$ the hyperplane bundle on $\PP^2$. On a Hirzebruch surface $\Sigma_m$ we denote by $F$ the class of fibre of the ruling,  $E$ the section with $E^2=-m$ and $H:=E+mF$. 

We use the notations about the Heisenberg algebra introduced in
\subsecref{Heisenberg} below: the elements $a_k$, $b_k$ with $k\in \Z$ are the generators of a Heisenberg algebra $\H$. The Fock space $F$ is an irreducible representation of $\H$ a vector space with basis consisting of vectors  $v_{\alpha,\beta}$, with $\alpha,\beta$ running through all partitions. 

\begin{theorem}
For $m\in \Z$ let 
$$H_m(t):= \sum_{k>0} b_{-k}b_k+ t\sum_{\|\mu\|=\|\nu\|-m}
a_{-\mu}a_{\nu}.$$
Then 
\begin{enumerate}
\item 
$$N^{d,\delta}(y)=
\left\<v_\emptyset\left|\Coeff_{t^d}\left[H_1(t)^{d(d+3)/2-\delta}\right]\right| v_{(1^d),\emptyset}\right\>,$$
\item 
$$
N^{(\Sigma_m,cF+dH),\delta}(y)=
\left\< v_{(1^c),\emptyset}\left| \Coeff_{t^d}\left[H_m(t)^{\binom{d+1}{2}m+dc+d+c-\delta}\right]\right|  v_{(1^{dm+c}),\emptyset}\right\>.
$$
\end{enumerate}
\end{theorem}

Roughly what happens is as follows:
There is a one to one correspondence between certain monomials in the Heisenberg algebra operators
and corresponding vertices of a floor diagram including numbers and weights of the ingoing and outgoing edges. 
The commutation relations in the Heisenberg algebra then correspond to the different ways how to connect the vertices to form a marked floor diagram, respecting the edges and their multiplicities. Thus the marked floor diagrams become the Feynman diagrams associated to these monomials in the Heisenberg algebra. The result then follows from a version of Wick's theorem which says that vacuum expectation values of Heisenberg operators on the Fock space can be computed in terms of Feynman diagrams.

This paper is organized as follow: in Section~\ref{Background} we
review the Heisenberg algebra, Fock space, and refined Severi
degrees. In Section~\ref{main_theorems}, we state our main result
(Theorem~\ref{NDEL}) relating refined Severi degrees and the Fock
space and various corollaries. We
introduce marked floor diagrams in
Section~\ref{sec:floor_diagrams}. In Section~\ref{Feynman}, we show
that marked floor diagrams are Feynman diagrams and prove a version of Wick's
theorem implying our results of Section~\ref{main_theorems}.

\section{Background}\label{Background}

\subsection{The Heisenberg algebra and the Fock space}\label{Heisenberg}
We review   the Fock space introduced in \cite{CoPa12}, changing some of the notations and conventions from there
and introducing a  $y$-deformation.
For $n\in \Z$ we define the \emph{quantum number} 
$$[n]_y:=\frac{y^{n/2}-y^{-n/2}}{y^{1/2}-y^{-1/2}}=y^{(n-1)/2}+\cdots+y^{-(n-1)/2}.$$
Note that $[n]_1=n$.
We consider the following $y$-deformed version $\H$ of the Heisenberg algebra modeled on the hyperbolic lattice. $\H$ is the Lie algebra over $\Q[y^{\pm 1/2}]$ generated by operators
$a_n,\ b_n$, for all $n\in \Z$, with the commutation relations
\begin{equation}\label{Heiscomm}[a_n,a_m]=[b_n,b_m]=0, \ [a_n,b_m]=[n]_{y}\delta_{n,-m}.
\end{equation}

The $a_{-n}$, $b_{-n}$  with $n>0$ are called \emph{creation operators}, the 
$a_{n}$, $b_{n}$  with $n>0$ are called \emph{annihilation operators}. We put $a_0=b_0:=0$.

The \emph{Fock space} $F$ is the free $\Q[y^{\pm 1/2}]$-module generated by the creation operators 
$a_{-n}, \ b_{-n}$, ($n>0$) acting on the so-called \emph{vacuum vector} $v_{\emptyset}\in F$. 
$F$ is an $\H$-module by requiring that $a_nv_{\emptyset}=b_nv_{\emptyset}=0$ for all $n>0$.
$F$ has therefore a $\Q[y^{\pm 1/2}]$-basis parametrized by pairs of partitions. 

We write partitions as $\mu=(1^{\mu_1}, 2^{\mu_2},\ldots)$, where
$\mu_i$ is the number of times $i$ occurs in $\mu$. We denote by
$\|\mu\|:=\sum_i i\mu_i$ the number partitioned by $\mu$ and $|\mu|=\sum_i \mu_i$ the length of the partition. We denote $\emptyset$ the empty partition of $0$.

For a partition $\mu$,
we define operators
\begin{equation}\label{abmu}a_\mu:=\prod_{i}\frac{(a_i)^{\mu_i}}{\mu_i!},\ a_{-\mu}:=\prod_{i}\frac{(a_{-i})^{\mu_i}}{\mu_i!},\ b_\mu:=\prod_{i}\frac{(b_i)^{\mu_i}}{\mu_i!},\quad b_{-\mu}:=\prod_{i}\frac{(b_{-i})^{\mu_i}}{\mu_i!}\in \H.\end{equation}
In particular $a_{\emptyset}=a_{-\emptyset}=1$, and similar for the $b_{\mu}$.
For partitions $\mu$, $\nu$ we define 
$v_{\mu,\nu}:=a_{-\mu}b_{-\nu}v_\emptyset$.
Then the $v_{\mu,\nu}$ form a $\Q[y^{\pm1/2}]$-basis of $F$.
A $\Q[y^{\pm1/2}]$-bilinear inner product $\< \, |\, \>$ is defined by
$\< v_\emptyset |v_\emptyset \>:=1$ and the condition that $a_n$ is adjoint to 
$a_{-n}$ and $b_n$ to $b_{-n}$. 
Explicitely this gives
\begin{equation}
\label{innprod}\<v_{\mu,\nu}|v_{\mu',\nu'}\>=\left(\prod_{i}\frac{([i]_y)^{\mu_i}}{\mu_i!}\right)\left(\prod_{j}\frac{([j]_y)^{\nu_j}}{\nu_j!} \right) \delta_{\mu,\nu'}\delta_{\nu,\mu'}.
\end{equation}
We also write $\<\alpha|A|\beta\>:=\<\alpha|A\beta\>$ for $\alpha,\beta\in F$, $A\in \H$.
If $A=\sum_n A_n t^n$ is an operator in $\H[[t]]$, we write 
$\<v_{\mu,\nu}|A| v_{\mu',\nu'}\>$ for the matrix element $\sum_{n} \<v_{\mu,\nu}|A_n| v_{\mu',\nu'}\>t^n.$
We write $\<A\>$ for the \emph{vacuum expectation value}
$\<v_\emptyset|A|v_\emptyset\>$.

\subsection{Refined Severi degrees}
\label{subsecRefinedSeveri}

We briefly review the definition of the refined Severi degrees form
\cite{BG14}, more details can be found there.

A \emph{lattice polygon} polygon $\Delta\subset \RR^2$ is a polygon
with vertices of
 integer coordiates.
The \emph{lattice length} of an edge $e$ of $\Delta$ is $\# e\cap \Z^2-1$. 
$\Delta$ is a closed subset of $\RR^2$. We denote by $int(\Delta)$, $\partial(\Delta)$ its interior and its boundary.
To a convex lattice polygon $\Delta$ one can associate a pair $S(\Delta)$, $L(\Delta)$ of a toric surface and 
a toric line bundle on $S(\Delta)$.  The toric surface is defined by the fan given by the outer normal vectors of $\Delta$. We have
$\dim H^0(S(\Delta),L(\Delta))=\# (\Delta\cap \Z^2)$. The arithmetic genus of a curve in $|L(\Delta)|$ is 
$g(\Delta)=\# (int(\Delta)\cap \Z^2)$.

Let $\Delta$ be a lattice polygon in $\RR^2$. A
non-zero vector $u \in \ZZ^2$ is \emph{primitive} if its entries are coprime.

\begin{definition}
A \emph{tropical curve of degree $\Delta$} is a continuous map $h: C \to
\RR^2$ satisfying:
\begin{enumerate}
\item$C$ is a abstract tropical curve (essentially a metric graph),
  possibly with multiple components.
\item $h(C)$ is a one-dimensional polyhedral complex with edges of
  rational slope, and non-negative integer weights $w(e)$ on all
  edges $e$, such that each vertex $V$ of $C$ is balanced, that is
\[
\sum_{e: \, V \in \partial e} w(e) \cdot v(V, e) = 0,
\]
where $v(V, e) \in \ZZ^2$ is the primitive vector starting at $V$ in
direction $e$.
\item For each primitive vector $u \in \ZZ^2$, the total weight
  of the unbounded edges in direction $u$ equals the lattice length of an edge of
  $\partial \Delta$ with outer normal vector $u$ (if there is no such edge, we require the total weight to be zero).
\end{enumerate}
\end{definition}

A tropical curve $(C, h)$ defines a \emph{dual subdivision}
of $\Delta$. (The edges of this subdivision are orthogonal to the
edges $e$ of $h(C)$ and have lattice length $w(e)$.) Each $3$-valent vertex $v$ of $h(C)$ corresponds to a
triangle $\Delta_v$ of the dual subdivision. 
We say that $(C, h)$ is \emph{simple} if all vertices of
  $C$ are $3$-valent, the self-intersections of $h$ are disjoint
  from vertices, and the inverse image under $h$ of self-intersection
  points consists of exactly two points of $C$.
 The \emph{number of nodes} of $(C, h)$ is the number of
  parallelograms of the dual subdivision if $(C, h)$ is
  simple. The lattice area $\Area(-)$ of a lattice polygon is twice
  its Euclidian area.

The \emph{refined multiplicity} of a simple
tropical curve $(C, h)$ is
\begin{equation}
\label{eqn:refinedMultiplicity}
\mult(C, h;y) = \prod_v  [\Area(\Delta_v)]_y,
\end{equation}
the product running over the $3$-valent vertices of $(C, h)$.

We now define the tropical refinement of Severi degrees.
We require the configuration of
tropical points to be in \emph{tropically generic position}; the
precise definition is given in \cite[Definition~4.7]{Mi05}. Roughly,
tropically generic means there are no tropical curves of unexpectedly small degree passing through
the points.

\begin{definition}
The \emph{refined
  Severi degree} $N^{\Delta, \delta}(y)$ 
is
\begin{equation}
\label{eqn:RefinedSeveriSum}
N^{\Delta, \delta}(y) := \sum_{(C, h)} \mult(C, h; y),
\end{equation}
where the sum is over all $\delta$-nodal tropical curves $(C, h)$
of degree $\Delta$
passing through $|\Delta \cap \ZZ^2| - 1 - \delta$ tropically generic points.
\end{definition}

By \cite[Theorem~1]{IM12} and \cite[Theorem~7.3]{BG14}, $N^{\Delta,
  \delta}(y)$ is independent of the generic point configuration.

There is a relative notion of refined Severi degrees if $\Delta$ has
at least one horizontal edge as is the case for $\PP^2$, $\Sigma_m$,
and $\PP(1,1,m)$. Again, we are brief, see Definitions 7.1 and 7.2
of \cite{BG14} for details. We now assume that $\Delta$ has a horizontal edge at
the bottom (i.e., with outer normal $(0,-1)$) of lattice length $d^b$. By a {\it sequence} we mean a collection $\alpha=(\alpha_1,\alpha_2,\ldots)$ of nonnegative integers, almost all of which are zero. For two sequences
$\alpha$, $\beta$ we define $\alpha!:=\prod_{i}\alpha_1!$,
$|\alpha|=\sum_i\alpha_i$, $\|\alpha\|=\sum_i i\alpha_i$, 
and
$\alpha+\beta=(\alpha_1+\beta_1,\alpha_2+\beta_2,\ldots)$.
Note that a sequence $\alpha=(\alpha_1,\alpha_2,\ldots)$ can be identified with the corresponding partition $(1^{\alpha_1},2^{\alpha_2},\ldots)$ and with this identification
the definitions of $\|\alpha\|$ and $|\alpha|$ correspond. 

Let $\alpha$ and $\beta$ be two sequences with $\|\alpha\| + \|\beta\| =
d^b$, let $D$ be a horizontal line very far below,
and let $\Pi$ be a tropically generic
point configuration of $|\Delta \cap \ZZ^2| -1 -\delta -\|\alpha\|-\|\beta\| +
|\alpha| + |\beta|$ points with precisely $|\alpha|$ points on
$D$.
A tropical curve $C$ passing through $\Pi$ is
  \emph{$(\alpha, \beta)$-tangent to $D$} if precisely $\alpha_i+\beta_i$ 
unbounded edges of $C$ are orthogonal to and intersect $D$ and have multiplicity
$i$ and, further,
$\alpha_i$ of the edges pass through $\Pi \cap D$.

\begin{definition}
The \emph{refined relative
  Severi degree} $N^{\Delta, \delta}(\alpha, \beta)(y)$ 
is the number of $\delta$-nodal
tropical curves $C$ of degree $\Delta$ passing through $\Pi$ that
are $(\alpha, \beta)$-tangent to $D$, counted with multiplicity
\[
\mult_{\alpha,\beta}(C;y) = \frac{1}{\prod_{i \ge 1}
  ([i]_y)^{\alpha_i}} \cdot \mult(C;y).
\]
\end{definition}

\subsection{$h$-transverse lattice polygons}
\label{sec:h-transverse_polygons}

We recall the definition of $h$-transverse lattice polygons from \cite{AB10}, with slightly different notations.
\begin{definition}
For us a {\it multiset} of integers is a tuple $(r_1^{i_1},\ldots,r_s^{i_s})$, where the $r_j$ are an ascending sequence of integers and 
the $i_j$ are positive integers. Thus they are the same as the maps
$\Z\to \Z_{\ge 0}$ 
with finite support.
For a multiset $r=(r_1^{i_1},\ldots,r_s^{i_s})$, we put $|r|=\sum_{j=1}^s i_j$ and $\|r\|=\sum_{j=1}^s i_j r_j.$
An \emph{ordering} of $r$ is a sequence $(a_1,a_2,\ldots,a_{|r|})$ of integers, with $i_j=\#\big\{ k\bigm| a_k=r_j\big\}$ for all $j$.
For instance the orderings of $(1^2,2)$ are $(1,1,2)$, $(1,2,1)$, $(2,1,1)$.
\end{definition}

\begin{definition}
A convex lattice polygon $\Delta\subset \RR^2$ is called $h$-transverse, if every edge has slope $0$, $\infty$ or $1/k$ for some integer $k$.

Now let $\Delta$ be an $h$-transverse convex lattice polygon.  Let  $d^t_\Delta$ and $d^b_\Delta$ be the lengths of the top and bottom edges of $\Delta$
(and $0$ if they do not exist).  Let  the slopes of outward normal vectors of the edges on the
right hand side
of $\Delta$ be (in ascending order)
$r_1,\ldots,r_s$ and let their lattice lengths be $i_1,\ldots,i_s$. Similarly let the slopes of the outward normal vectors on the left hand side be (in ascending order) $l_1,\ldots,l_t$ and let their lattice lengths be $j_1,\ldots,j_t$.
Then 
$$r_\Delta:=(r_1^{i_1},\ldots, r_s^{i_s}), \quad l_{\Delta}:=(l_1^{j_1},\ldots, l_t^{j_t})$$
  are multisets
with
$d^b_\Delta+\| l_\Delta\| =d^t_\Delta+ \| r_\Delta\|$, and $h_\Delta:=|l_\Delta|=|r_\Delta|$ is the height of $\Delta$.

Conversely, given $d^b, d^t\in \Z_{\ge 0}$, $r,l: \Z\to \Z_{\ge 0}$ with finite support and $| l|=|r|$, such that
$d^t+\| r\|:=d^b+\|l\|$,
 there is an $h$-transverse lattice polygon
$\Delta$ such that $d^b=d^b_\Delta$, $d^t=d^t_\Delta$, $l=l_\Delta$, $r=r_\Delta$.
\end{definition}

\begin{example}\label{htrex}
We list some examples of toric surfaces and line bundles corresponding to $h$-transverse polygons.
\begin{enumerate}
\item For $S=\PP^2$, $L=\O(d)$, we have 
$d^t_\Delta=0$, $d_\Delta^b=d$, $l_\Delta=(0^d)$ and $r_\Delta=(1^d)$.
\item For $S=\Sigma_m$ a rational ruled surface, let $F$ be the class of a fibre of the ruling; let $E$ be the section with self intersection $-m$; Let $H=E+mF$. 
For $L=dH+cF$ we have 
$d^t_\Delta=c$, $d^b=c+dm$, $l_\Delta=(0^d)$, $r_\Delta=(m^d)$.
\item 
For weighted projective space $\PP(1,1,m)$, and $L=dH$ with $H$ the hyperplane bundle with $H^2=m$, we have
$d^t_{\Delta}=0$, $d^b_\Delta=dm$, $l_\Delta=(0^d)$ and $r_\Delta=(m^d)$. As $dH$ on $\Sigma_m$ is the pullback of the line bundle with the same name on 
$\PP(1,1,m)$,  this in fact is the case $c=0$ of the previous example. 
 \item 
 For a weighted projective space $\PP(1,m-1,m)$ and $L=dH$ where $H$ is the hyperplane bundle with $H^2=m(m-1)$, we have
 $d^t_\Delta=0$, $d^b_\Delta=0$, $l_\Delta=(0^{dm})$, 
 $r_\Delta=((-1)^{d(m-1)},(m-1)^d)$.
\end{enumerate}

\vskip-1cm
\begin{figure}[htbp]
 \begin{tabular}{cccc}
\raisebox{15.5pt}{\includegraphics[scale = 0.067]{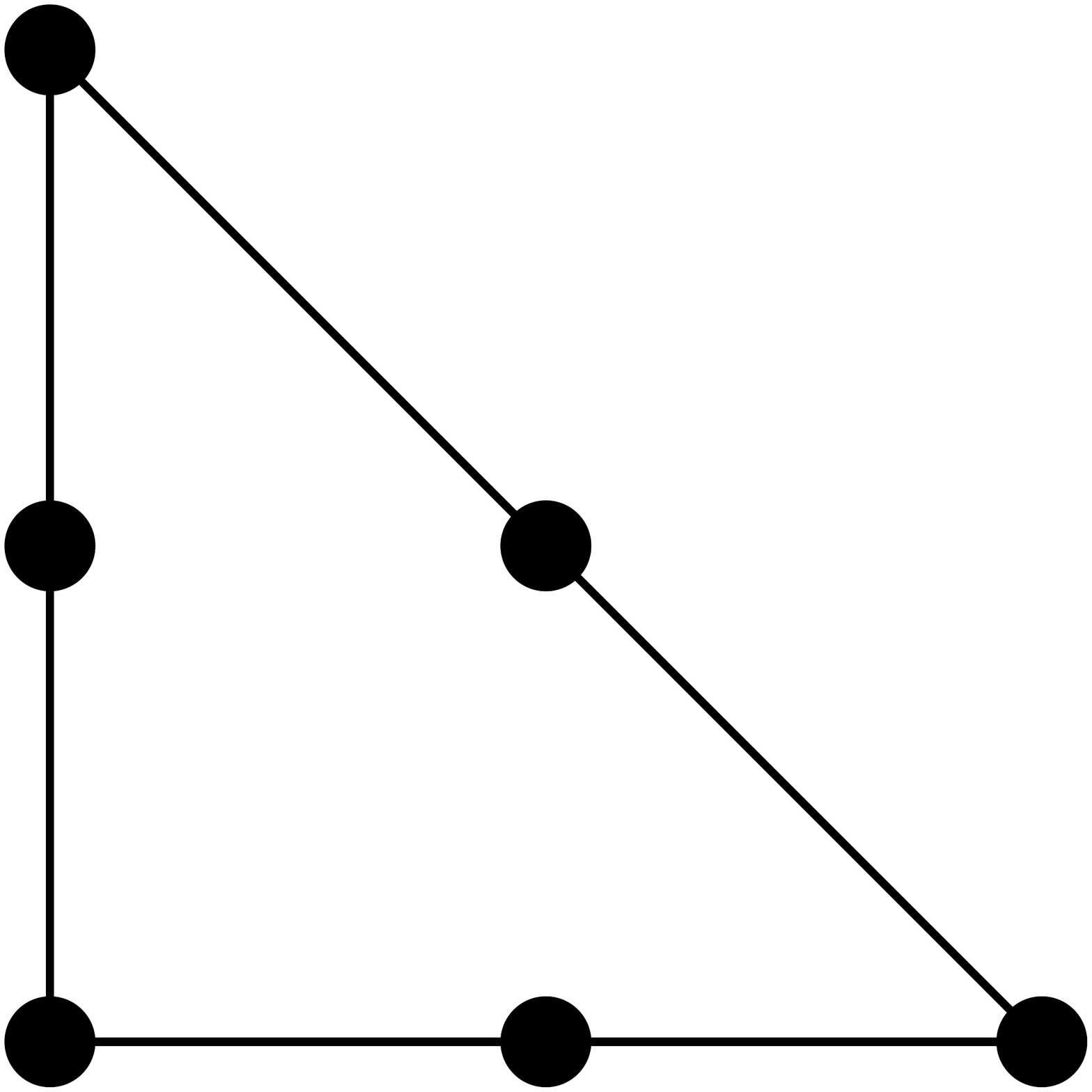}}  \quad
&
\quad 
\raisebox{-35pt}{
\includegraphics[scale = 0.2]{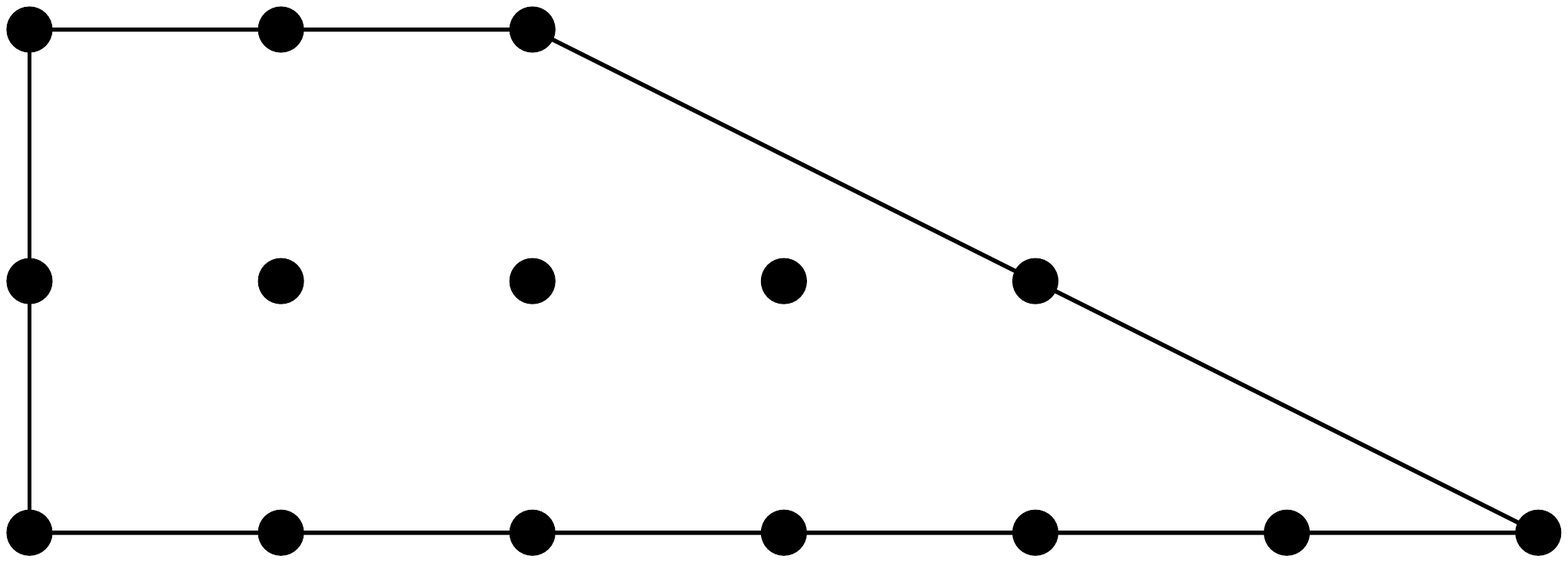}} \quad 
&
\quad 
\raisebox{-6.5pt}{
 \includegraphics[scale = 0.13]{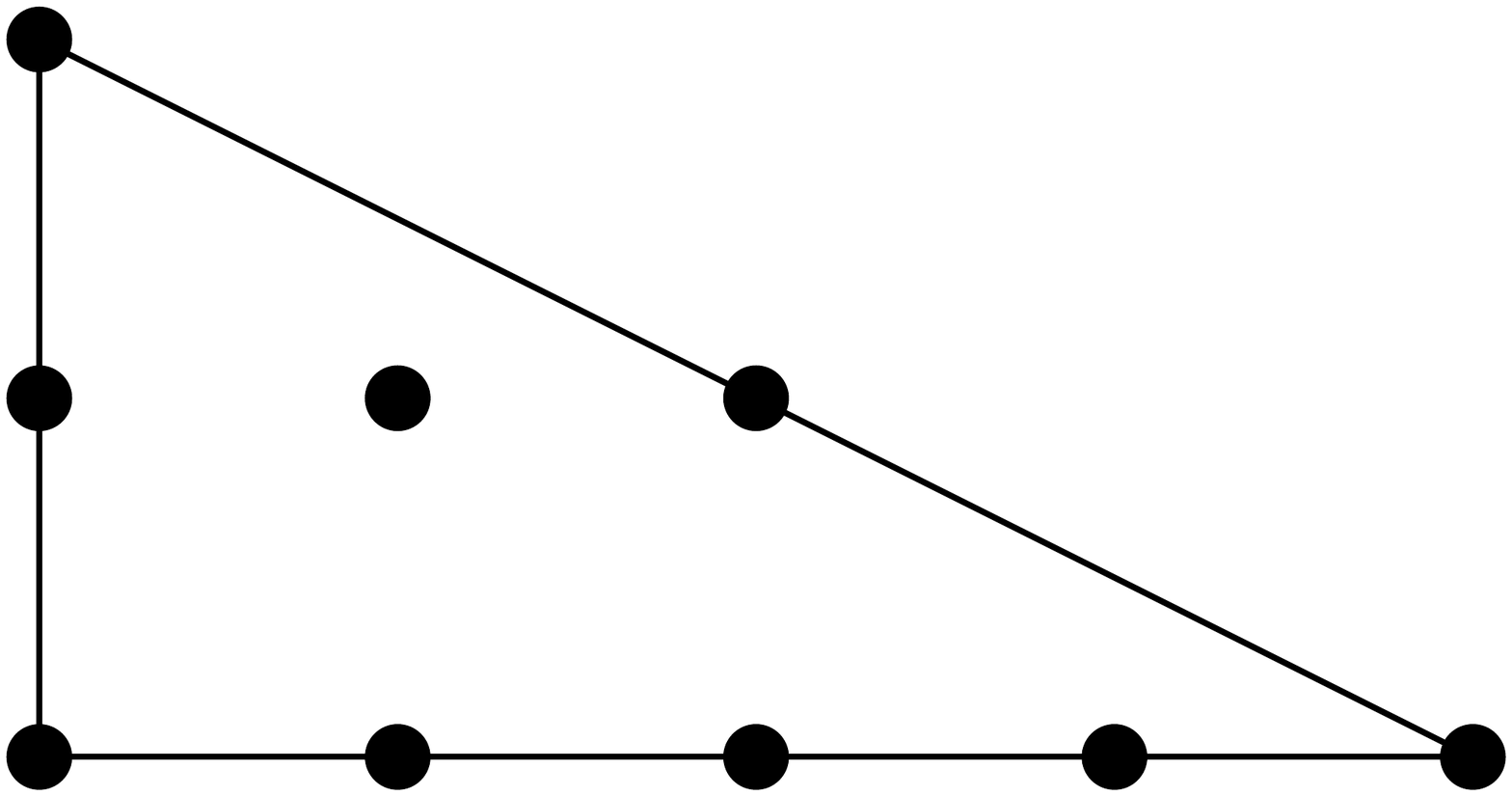}} \quad
&
\quad 
\raisebox{7.5pt}{
 \includegraphics[scale = 0.075]{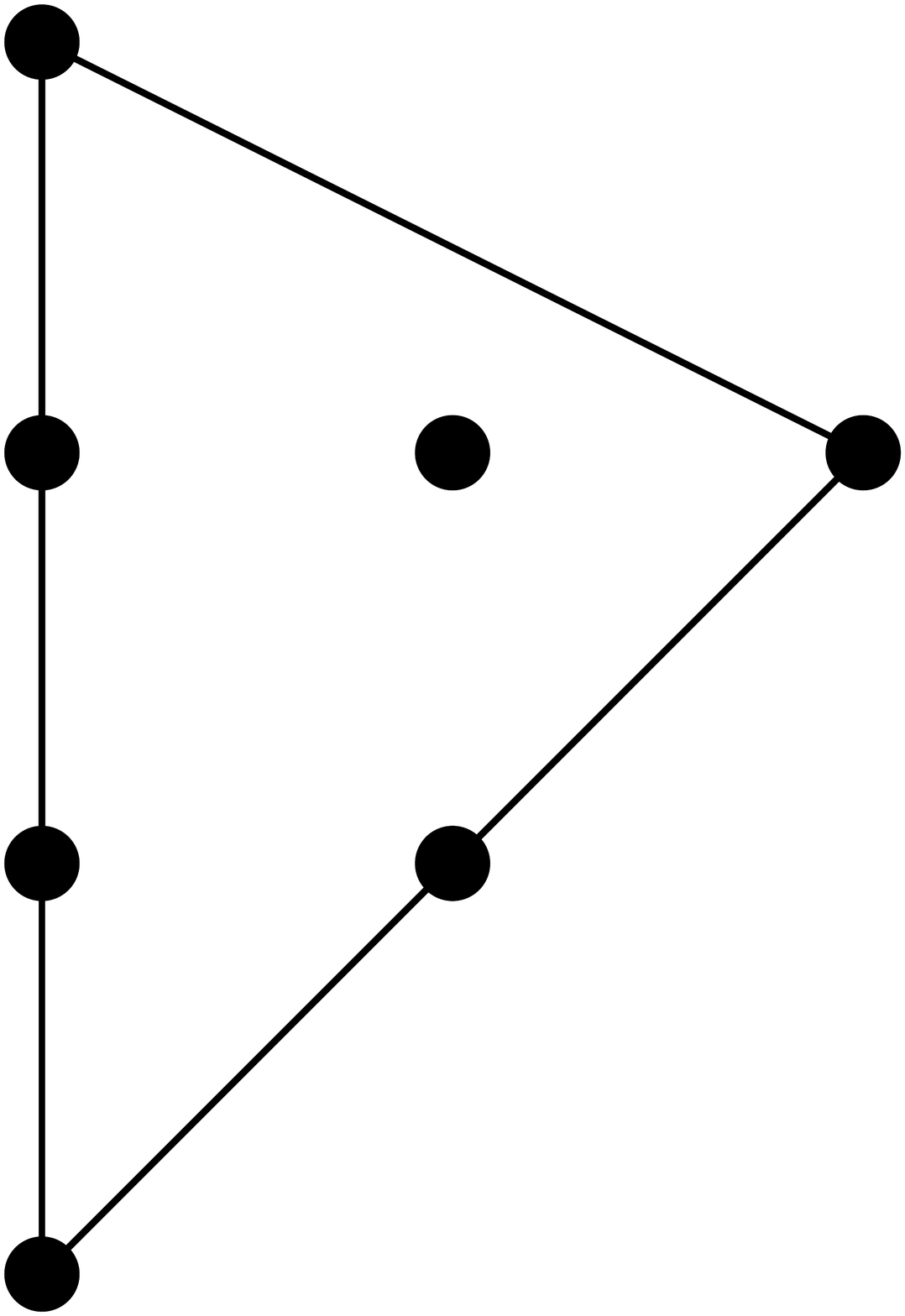}} \quad \\
\vspace{-4em}
 \end{tabular}
\caption{The lattice polygons of $(\PP^2,2H)$, $(\Sigma_2, 2F+2H)$,
  $(\PP(1,1,2),2H)$, and
  $\PP(1,2,3),H)$.}
\label{fig:HirzebruchCurve}
\end{figure}
\end{example}

\section{Main theorems}
\label{main_theorems}
We first state our main result for general $h$-transverse lattice polygons, relating the refined Severi degrees to matrix elements of operators in the Heisenberg algebra $\H$.
The general formula is a bit complicated, but as corollaries we get somewhat more attractive formulas for the surfaces of  \exmpref{htrex}.

\begin{remark}\label{grad}
The Fock space $F$ has a grading $F=\bigoplus_{n\ge 0} F_n$ with $F_n$ the span of the $v_{\mu,\nu}$ with $\|\mu\|+\|\nu\|=n$.
We denote by $\widehat F$ the completion of $F$ with respect to this grading, i.e. elements of $\widehat F$ are 
possibly infinite sums $\sum_{n\in \Z_{\ge 0}} v_n$, with $v_n\in F_n$.

Then $\H$ is  a graded algebra: $\H=\sum_{n\in \Z} \H_n$. This grading is defined by giving degree $n\in \Z$ to $a_{-n}$ and $b_{-n}$. 
It coincides with the degree as operators on $F$: elements of $\H_n$ send $F_m$ to $F_{m+n}$. We denote  $\widehat \H$ the set of linear maps $f:\widehat F\to \widehat F$, which are expressible as possibly infinite sums 
$f=\sum_{n\ge 0} h_n$ with $h_n\in \H$. 
\end{remark}

\begin{notation}
 Let $T_i, \ i\in \Z$, be noncommuting variables (with no relations). 
For a finite sequence $I=(i_1,\ldots,i_n)$ of integers, we put $T^I:=T_{i_1}T_{i_2}\ldots T_{i_n}$.
For a ring $R$, let $R\{T\}$ be the set of finite linear combinations $\sum_I a_I T^I$, with $I$ running through finite sequences of integers 
with coefficientwise addition, and multiplication defined by concatenation: 
$$(T_{i_1}\ldots T_{i_n})(T_{j_1}\ldots T_{j_m})=T_{i_1}\ldots T_{i_n}T_{j_1}\ldots T_{j_m}.$$
For $M=\sum_{I}a_IT^I\in R\{T\}$, with the $I$ distinct, we put $\Coeff_{T^I}[M]=a_I$.

We define an operator on the Fock space $F$ by 
$$H(T):=\sum_{k>0} b_{-k}b_{k}+\sum_{\mu,\nu} a_{-\mu}a_{\nu}T_{\|\nu\|-\|\mu\|}\in \widehat \H\{T\},$$
where $\mu, \nu$ run through all partitions. 
For sequences $I=(i_1,\ldots,i_n),\ J=(j_1,\ldots,j_n)$, as usual 
$I-J=(i_1-j_1,\ldots, i_n-j_n).$
\end{notation}

\begin{notation}
In the future we will write $\#\Delta$ instead of $\#(\Delta\cap \Z^2)$ for the lattice points in a convex lattice polygon.
Below we will usuall write $d^b$ and $d^t$ instead of $d^b_\Delta$, $d^t_\Delta$, when this does not lead to confusion.
\end{notation}

\begin{notation}
We write $I_y^{\alpha+\beta}=\prod_{i}
 [i]_y^{\alpha_i+\beta_i}$ for finite sequences $\alpha$ and
 $\beta$. 
\end{notation}

The main theorem of this paper describes the refined Severi degrees of $h$-transversal lattice polygons in terms of matrix elements of operators on Fock space.
\begin{theorem}\label{NDEL}
Let $\Delta$ be an $h$-transverse lattice polygon. 
\begin{eqnarray}
\label{NDEL1}N^{\Delta,\delta}(y)&=&\left\< v_{(1^{d^t}),\emptyset}\left| \sum_{R,L} \Coeff_{T^{R-L}}\left[
H(T)^{\#\Delta-\delta-1}\right]\right|v_{(1^{d^b}),\emptyset}\right\>,\\
\label{NDEL2}\qquad N^{\Delta,\delta}(\alpha,\beta)(y)&=&\frac{\alpha!}{I_y^{\alpha+\beta}}\left\< v_{(1^{d^t}),\emptyset}\left| \sum_{R,L} \Coeff_{T^{R-L}}\left[
H(T)^{\#\Delta-\delta-1-d^b+|\beta|}\right]\right|v_{\beta,\alpha}\right\>.
\end{eqnarray}
Here $R$ and $L$ run through the orderings of $r_{\Delta}$ and $l_{\Delta}$ respectively, and in the second formula $\alpha$, $\beta$ are partitions with 
$\|\alpha\|+\|\beta\|=d^b.$
\end{theorem}

If $\Delta$ has only one left direction, i.e. we can write $l_\Delta=(l^{h_{\Delta}})$ for some $l\in \Z$, the formula simplifies, and we do not need to use noncommutative variables anymore. 
Let $(t_i)_{i\in \Z}$ be commuting variables.
For a multiset  $I=(l_1^{i_1},\ldots,l_s^{i_s})$ of integers, we put $\mathbf t^I:=t_{l_1}^{i_1}\ldots t_{l_s}^{i_s}$. For a ring $R$ we write $R[\mathbf t]$  for the ring of polynomials series in the $(t_i)_{i\in \Z}.$
We define an operator on the Fock space $F$ by 
$$H(\mathbf t):=\sum_{k>0} b_{-k}b_{k}+\sum_{\mu,\nu} a_{-\mu}a_{\nu}t_{\|\nu\|-\|\mu\|}\in \widehat \H[\mathbf t],$$
where $\mu, \nu$ run through all partitions. 
If $r_{\Delta}=(r_1^{n_1},\ldots,r_s^{n_s})$, we write $(r_{\Delta}-l):=((r_1-l)^{n_1},\ldots,(r_s-l)^{n_s})$

\begin{corollary}\label{oneleft}
Let $\Delta$ be an $h$-transverse lattice polygon with $l_{\Delta}=(l^{h_{\Delta}})$.
Then 
$$
N^{\Delta,\delta}(y)=\left\< v_{(1^{d^t}),\emptyset}\left| \Coeff_{\mathbf t^{(r_\Delta-l)}}\left[
H(\mathbf t)^{\#\Delta-\delta-1}\right]\right|v_{(1^{d^b}),\emptyset}\right\>.$$
$$N^{\Delta,\delta}(\alpha,\beta)(y)=\frac{\alpha!}{I_y^{\alpha+\beta}}\left\< v_{(1^{d^t}),\emptyset}\left|  \Coeff_{\mathbf t^{(r_{\Delta}-l)}}\left[
H(\mathbf t)^{\#\Delta-\delta-1-d^b+|\beta|}\right]\right|v_{\beta,\alpha}\right\>,$$
where $\alpha$, $\beta$ are partitions with $\|\alpha\|+\|\beta\|=d^b.$
\end{corollary}

In \exmpref{htrex} the assumptions of \corref{oneleft} are fullfilled.

For an integer $m$ let 
\begin{align*}H_m(t)&:= \sum_{k>0} b_{-k}b_k+ t\sum_{\|\mu\|=\|\nu\|-m}
a_{-\mu}a_{\nu}\end{align*}

\begin{corollary}
\begin{enumerate}\label{hexcor}
\item For $\PP^2$ we have
\begin{align*}N^{d,\delta}(y)&=\left\<v_{\emptyset} \left| \Coeff_{t^{d}}\left[H_1(t)^{d(d+3)/2-\delta}\right]\right| v_{(1^d),\emptyset}\right\>,\\
N^{d,\delta}(\alpha,\beta)(y)&=\frac{\alpha!}{I_y^{\alpha+\beta}}\left\<v_{\emptyset} \left| \Coeff_{t^{d}}\left[H_1(t))^{d(d+3)/2-\delta-d+|\beta|}\right]\right| v_{\beta,\alpha}\right\>,\\
\end{align*}
for $\alpha,\beta$ partitions with $\|\alpha\|+\|\beta\|=d$.
\item In the case of rational ruled surfaces $\Sigma_m$ and the weighted projective space 
$\PP(1,1,m)$, we get 
 \begin{align*}
N^{(\Sigma_m,cF+dH),\delta}(y)&=
\left\< v_{(1^c),\emptyset}\left| \Coeff_{t^d}\left[H_m(t)^{\binom{d+1}{2}m+cd+c+d-\delta}\right]\right|  v_{(1^{dm+c}),\emptyset}\right\>,\\
N^{(\Sigma_m,cF+dH),\delta}(\alpha,\beta)(y)&=\frac{\alpha!}{I_y^{\alpha+\beta}}
\left\< v_{(1^c),\emptyset}\left| \Coeff_{t^d}\left[H_m(t)^{\binom{d+1}{2}m+cd+c+d-\delta-\|\alpha\|-\|\beta\|+|\beta|}\right]\right|  v_{\beta,\alpha}\right\>,\\
N^{(\PP(1,1,m),dH),\delta}(y)&=\left\< v_\emptyset\left|\Coeff_{t^d}\left[H_m(t)^{\frac{m}{2}d^2+(\frac{m}{2}+1)d-\delta}\right]\right|  v_{(1^{dm}),\emptyset}\right\>,\\
N^{(\PP(1,1,m),dH),\delta}(\alpha,\beta)(y)&=\frac{\alpha!}{I_y^{\alpha+\beta}}\left\< v_\emptyset \left|\Coeff_{t^d}\left[H_m(t)^{\frac{m}{2}d^2+(\frac{m}{2}+1)d-\delta-\|\alpha\|-\|\beta\|+|\beta|}\right]\right|  v_{\beta,\alpha}\right\>.
\end{align*}
\item
Let 
$$G_m(t):=\sum_{k>0} b_{-k}b_k +\sum_{\|\mu\|=\|\nu\|+1}
a_{-\mu}a_{\nu}+t\sum_{\|\mu\|=\|\nu\|-m}
a_{-\mu}a_{\nu}.$$
Then for the surface $\PP(1,(m-1),m)$ we get 
\begin{align*}
N^{(\PP(1,(m-1),m),dH),\delta}(y)&=\left\< \Coeff_{t^d}\left[G_{m-1}(t)^{\binom{m}{2}d^2+md-\delta}\right]\right\>.
\end{align*}

\end{enumerate}
\end{corollary}

\begin{remark} \label{grading}
\begin{enumerate}
\item By definition $\Coeff_{t^d} H_m(t)^N$ has degree $-dm$ with respect to the grading on $F$. Thus $\<v_{\alpha',\beta'}|\Coeff_{t^d}H_m(t)^N|v_{\alpha,\beta}\>=0$ unless 
$\|\alpha\|+\|\beta\|-\|\alpha'\|-\|\beta'\|=dm$.
Therefore the formulas in (1) and (2) (but not (3)) of \corref{hexcor} are also true without taking the coefficient of $t^d$.

\item In the completion $\widehat F$ of $F$, we have the easy identity
$\exp(a_{-1})v_\emptyset=\sum_{m\ge 0} v_{((1^m),\emptyset)}$.
Therefore the argument of  (1) shows also that 
\begin{align*}\left\<v_{(1^c),\emptyset}\left | \Coeff_{t^d}H_m(t)^N\right| v_{(1^{c+dm}),\emptyset}\right\>&=\sum_{n\ge 0}\left\<v_{(1^c),\emptyset}\left | \Coeff_{t^d}H_m(t)^N\right| v_{(1^{n}),\emptyset}\right\>\\
&=\left\<v_{(1^c),\emptyset}\left | \Coeff_{t^d}H_m(t)^N\right| \exp(a_{-1})v_\emptyset \right\>.
\end{align*}

\end{enumerate}
\end{remark}

\begin{corollary} \label{genfun}
One can  organize the above results into generating functions:
\begin{enumerate}
\item  For $\PP^2$ we have
$$\sum_{d\ge 0}\sum_{\delta\ge 0}
\frac{t^d q^{d(d+3)/2-\delta}}{(d(d+3)/2-\delta)!}N^{d,\delta}(y)=\left\<\exp(qH_1(t))\exp(a_{-1})\right\>.$$
\item On $\Sigma_m$ and $\PP(1,1,m)$ we get
\begin{align*}\sum_{c\ge 0}\sum_{d\ge 0}\sum_{\delta\ge 0}
\frac{s^c t^d q^{\binom{d+1}{2}m+cd+c+d-\delta}}{(\binom{d+1}{2}m+cd+c+d-\delta)!}N^{(\Sigma_m,cF+dH),\delta}(y)&=\left\<\exp(a_1 s)\exp(qH_m(t))\exp(a_{-1})\right\>,\\
\sum_{d\ge 0}\sum_{\delta\ge 0}
\frac{t^d q^{\frac{m}{2}d^2+(\frac{m}{2}+1)d-\delta}}{(\frac{m}{2}d^2+(\frac{m}{2}+1)d-\delta)!}N^{(\PP(1,1,m),dH),\delta}(y)&=\left\<\exp(qH_m(t))\exp(a_{-1})\right\>.
\end{align*}
\item On $\PP(1,m-1,m)$ we get
$$\sum_{d\ge 0}\sum_{\delta\ge 0}
\frac{t^d q^{\binom{m}{2}d^2+md-\delta}}{(\binom{m}{2}d^2+md-\delta)!}N^{(\PP(1,(m-1),m),dH),\delta}(y)=\left\<\exp(qG_{m-1}(t))\right\>.$$
\end{enumerate}
\end{corollary}

\thmref{NDEL} will be proven in \secref{Feynman}. In the rest of this section we will deduce \corref{oneleft}, \corref{hexcor}, \corref{genfun} from \thmref{NDEL}.

\begin{proof}[Proof of \corref{oneleft}] 
The map $T_i\mapsto t_i$ gives a homomorphism $\phi:\widehat
\H\{T\}\to \widehat \H[\mathbf t]$.
Clearly  $\phi(H(T)^N)=H(\mathbf t)^N$ for all $N$. 
It is also clear by definition that we can write 
$H(T)^N=\sum_{R} h_R T^R$, where the coefficient $h_R\in \widehat \H$ does not depend on the sequence $R$, but only on the multiset of which $R$ is an ordering. Thus we can write
$$H(T)^N=\sum_r \sum_{R} h_r T^R,\qquad H(\mathbf t)^N=\sum_{r}\Big(\sum_{R} h_r\Big) \mathbf t^r,$$
where the outer sums are over multisets $r$ of integers and the inner sums are over the orderings $R$ of $r$, and the $h_r$ are elements of $\widehat \H$. 
Thus 
\begin{equation}\label{Hcoeff}\sum_R \Coeff_{T^R} [H(T)^N]=\sum_{R} h_r=\Coeff_{\mathbf t^r}[H(\mathbf t)^N], 
\end{equation} where again the sums are over the orderings $R$ of $r$.

We are assuming $l_\Delta=(l^{h_\Delta})$. Thus $l_\Delta$ has only the unique ordering $L:=(l,\ldots,l)$
Thus  $R=(r_1,\ldots r_{h_\Delta})\mapsto R-L:=(r_1-l,\ldots, r_{h_\Delta}-l)$ is a bijection of the orderings of  $r_\Delta$ to those of $(r_\Delta-l)$.
Therefore for any $N$
$$\sum_{R,L}\Coeff_{T^{R-L}}\left[H(T)^N\right]=\sum_{R'}\Coeff_{T^{R'}} \left[H(T)^N\right]=\Coeff_{\mathbf t^{(r_\Delta-l)}}\left[H(\mathbf t)^N\right],$$ where the first sums are over the reorderings of $r_\Delta$ and $l_\Delta$,the second sum is over the reorderings of $(r_\Delta-l)$, and the last identity is by \eqref{Hcoeff}. 
Thus  \corref{oneleft} follows from \thmref{NDEL}.
\end{proof}

\begin{proof}[Proof of \corref{hexcor}] In all the cases we have $l=0$, thus   $(r_{\Delta}-l):=r_{\Delta}=(r_1^{n_1},\ldots,r_s^{n_s})$. Then the right hand side of the formulas of \corref{oneleft} does not change if we set $t_i=0$ whenever $i\not \in \{r_1, \ldots, r_s\}$. 
By \exmpref{htrex} we have the following:

(1) For $S=\PP^2$, $L=dH$, we have $r_{\Delta}=(1^d)$, $d_{\Delta}^t=0$, $d_\Delta^b=d$ and it is easy to see that $\#\Delta-1=d(d+3)/2$. 
Thus the formulas follow at once by setting $t_i:=0$ for $i\ne 1$, $t_1:=t$ in \corref{oneleft}.

(2)  For $S=\Sigma_m$, $L=cF+dH$, we have $r_{\Delta}=(m^d)$, $d_{\Delta}^t=c$, $d_\Delta^b=c+md$ and it is easy to see that $\#\Delta-1=\binom{d+1}{2}m+cd+c+d$. 
Thus the formula follows  by setting $t_i:=0$ for $i\ne m$, $t_m:=t$ in \corref{oneleft}. The formula for $\PP(1,1,m)$ follows by setting $c=0$.

(3)  For $S=\PP(1,m-1,m)$, $L=dH$ we have $r_{\Delta}=((-1)^{(m-1)d},(m-1)^d)$, $d_{\Delta}^t=d_\Delta^b=0$. 
Dividing $\Delta$ horizontally into two triangles, congruent to the polygons  for $(\PP^2,d(m-1)H)$ and for $(\PP(1,1,m-1),dH)$, we compute $\#\Delta-1=\binom{m}{2}d^2+md$. 
Thus setting $t_i:=0$ for $i\not \in  \{-1,(m-1)\}$, $t_{-1}:=s$, $t_{m-1}:=t$ in \corref{oneleft}, we get
\begin{equation} \label{steq} N^{(\PP(1,(m-1),m),dH),\delta}(y)=\left\< \Coeff_{s^{(m-1)d}t^d}\left[G_{m-1}(s,t)^{\binom{m}{2}d^2+md-\delta}\right]\right\>,
\end{equation}
with 
$$G_m(s,t):=\sum_{k>0} b_{-k}b_k +s\sum_{\|\mu\|=\|\nu\|+1}
a_{-\mu}a_{\nu}+t\sum_{\|\mu\|=\|\nu\|-m}
a_{-\mu}a_{\nu}.$$ 
Note however that the coefficient of $s$ of $G_{m-1}(s,t)$ has degree $1$ and the coefficient of $t$ has degree $-(m-1)$. 
Thus $\<\Coeff_{s^nt^d}[G_{m-1}(s,t)]\>=0$ unless $n=(m-1)d$, and thus
 \eqref{steq} remains true if we put $s=1$.
\end{proof}

\begin{proof}[Proof of \corref{genfun}]
(1) By \corref{hexcor} and \remref{grading} we have
$$N^{d,\delta}(y)=\left\<v_\emptyset \left| H_1(t)^{d(d+3)/2-\delta}\right|\exp(a_{-1})v_\emptyset \right\>$$
  for all $\delta\in \Z$, and both sides of this equation vanish if $d(d+3)/2-\delta<0$. Writing $n:=d(d+3)/2-\delta$, we get therefore
\begin{align*}
\sum_{d\ge 0}&\sum_{\delta\ge 0}\frac{t^d q^{d(d+3)/2-\delta}}{(d(d+3)/2-\delta)!}N^{d,\delta}(y)\\&=
\sum_{d\ge 0}t^d\sum_{\delta\ge 0}\left\<v_\emptyset\left|\Coeff_{t^d}\left[\frac{(qH_1(t))^{d(d+3)/2-\delta}}{(d(d+3)/2-\delta)!}\right]\right| \exp(a_{-1})v_{\emptyset}\right\>\\
&=\sum_{n \ge 0}\sum_{d\ge 0}t^d\left\<v_\emptyset\left|\Coeff_{t^d}\left[\frac{(qH_1(t))^{n}}{n!}\right]\right| \exp(a_{-1})v_{\emptyset}\right\>=\<\exp(qH_1(t))\exp(a_{-1})\>.
\end{align*}
(2) is similar. Using again  \corref{hexcor} and \remref{grading} we have in the same way as before
\begin{align*}
&\sum_{c\ge 0}\sum_{d\ge 0}\sum_{\delta\ge 0}
\frac{s^c t^d q^{\binom{d+1}{2}m+cd+c+d-\delta}}{(\binom{d+1}{2}m+cd+c+d-\delta)!}N^{(\Sigma_m,cF+dH),\delta}(y)\\
&=
\sum_{c\ge 0}\sum_{d\ge 0}t^d\sum_{\delta\ge 0}\left \<v_{(1^c),\emptyset}s^c\left| \Coeff_{t^d}\left[\frac{(qH_m(t))^{(\binom{d+1}{2}m+cd+c+d-\delta)}}{(\binom{d+1}{2}m+cd+c+d-\delta)!}\right]
\right| \exp(a_{-1})v_{\emptyset}\right\>\\&=\sum_{c\ge 0}\sum_{n\ge 0}\left \<v_{(1^c),\emptyset}s^c\left| \frac{(qH_m(t))^{n}}{n!}
\right| \exp(a_{-1})v_\emptyset\right\>\\&
=\left<\exp(a_{-1}s)v_\emptyset\left| \exp(qH_m(t)) \right| \exp(a_{-1})v_\emptyset\right\>=\left\< \exp(a_{1}s)\exp(qH_m(t))  \exp(a_{-1})\right\>,
\end{align*}
where in the last step we use that $a_1$ is the adjoint of $a_{-1}$.

(3) By \corref{hexcor}  and using the same arguments as in (1) we get 
\begin{align*}
&\sum_{d\ge 0}\sum_{\delta\ge 0}\frac{t^dq^{\binom{m}{2}d^2+md-\delta}}{(\binom{m}{2}d^2+md-\delta)!}N^{(\PP(1,(m-1),m),dH),\delta}(y)\\
&=\sum_{d\ge 0}t^d\sum_{\delta\ge 0}\left\< \Coeff_{t^d}\left[\frac{(qG_{m-1}(t))^{\binom{m}{2}d^2+md-\delta}}{(\binom{m}{2}d^2+md-\delta)!}\right]\right\>.\\
&=\sum_{d\ge 0}t^d\sum_{n\ge 0}\left\< \Coeff_{t^d}\left[\frac{(qG_{m-1}(t))^{n}}{n!}\right]\right\>=\left\<\exp(qG_{m-1}(t))\right\>.
\end{align*}
\end{proof}

\begin{remark} 
In \cite{GS12}, \cite{BG14} irreducible refined Severi degrees $N_0^{(S,L),\delta}(y)$, $N^{\Delta,\delta}_0(y)$ are introduced and studied.
They give a count of irreducible tropical curves:
$$N^{\Delta,\delta}_0=\sum_{(C,h)} \mult(C,h,y),$$ where the sum is now over {\em irreducible} $\delta$-nodal tropical curves through 
$|\Delta\cap \Z|-1-\delta$ tropically generic points. At $y=1$ they specialize to the irreducible Severi degrees 
counting irreducible complex curves of degree $\Delta$ and at $y=-1$ they specialize to the irreducible Welschiger invariants, 
counting irreducible real curves.

\corref{genfun} also provides a generating function for the $N_0^{(S,L),\delta}$, as follows.
By \cite{GS12}, \cite{BG14} we have the formula
\begin{align*}
\sum_{L,\delta} \frac{z^{\dim|L|-\delta}}{(\dim|L|-\delta)!} v^L N_0^{(S,L),\delta}(y)=
\log\left(\sum_{L,\delta}\frac{z^{\dim|L|-\delta}}{(\dim|L|-\delta)!} v^L   N^{(S,L),\delta}(y)\right)\!,
\end{align*}
Here $\big\{v^L\big\}_L$, are elements of the Novikov ring i.e. $v^{L_1}v^{L_2}=v^{L_1+L_2}$. On $\PP^2$,   $\PP(1,1,m)$, $\PP(1,m-1,m)$, $L$ runs through the $nH$ with $n\ge 0$, and on $\Sigma_m$, $L$ runs through $dH+cF$ with $c,d\ge 0$.

Combining this with \corref{genfun}  we get the formulas
\begin{align*}&\sum_{d\ge 0}\sum_{\delta\ge 0}
\frac{t^d q^{d(d+3)/2-\delta}}{(d(d+3)/2-\delta)!}N_0^{d,\delta}(y)=\log\left(\left\<\exp(qH_1(t))\exp(a_{-1})\right\>\right),\\
&\sum_{c\ge 0}\sum_{d\ge 0}\sum_{\delta\ge 0}
\frac{s^c t^d q^{\binom{d+1}{2}m+cd+c+d-\delta}}{(\binom{d+1}{2}m+cd+c+d-\delta)!}N_0^{(\Sigma_m,cF+dH),\delta}(y)\\
&\qquad\qquad\qquad\qquad\qquad\qquad\qquad\qquad=\log\left(\left\<\exp(a_1 s)\exp(qH_m(t))\exp(a_{-1})\right\>\right),\\
&\sum_{d\ge 0}\sum_{\delta\ge 0}
\frac{t^d q^{\frac{m}{2}d^2+(\frac{m}{2}+1)d-\delta}}{(\frac{m}{2}d^2+(\frac{m}{2}+1)d-\delta)!}N_0^{(\PP(1,1,m),dH),\delta}(y)=\log\left(\left\<\exp(qH_m(t))\exp(a_{-1})\right\>\right),\\
&\sum_{d\ge 0}\sum_{\delta\ge 0}
\frac{t^d q^{\binom{m}{2}d^2+md-\delta}}{(\binom{m}{2}d^2+md-\delta)!}N_0^{(\PP(1,(m-1),m),dH),\delta}(y)=\log\left(\left\<\exp(qG_{m-1}(t))\right\>\right).
\end{align*}
\end{remark}

\begin{remark}
\corref{genfun} can be easily extended to relative refined Severi degrees.
For partitions $\alpha$, $\beta$, write $u^\alpha:=\prod_i u_i^{\alpha_i}$, $w^\beta:=\prod_i w_i^{\beta_i}$. Then we have the easy identity
$$\exp\left(\sum_{n>0}\frac{1}{[n]_y} \left(b_{-n}u_n+a_{-n}w_n\right)\right)v_\emptyset
=\sum_{{\alpha,\beta}\atop{\text{partitions}}}\frac{v_{\beta,\alpha}}{I_y^{\alpha+\beta}}u^\alpha w^\beta.$$
Using this, the same arguments as in the proof of \corref{genfun} show for instance
\begin{align*}
\sum_{d\ge 0} \sum_{\delta\ge 0}\sum_{\alpha,\beta}&
\frac{t^d q^{d(d+3)/2-\delta-d+|\beta|}N^{d,\delta}(\alpha,\beta)(y)}{(d(d+3)/2-\delta-d+|\beta|)!}  \frac{u^{\alpha}}{\alpha!}w^\beta\\&=
\left\<\exp(qH_1(t))\exp\left(\sum_{n>0}\frac{u_{n}b_{-n}+w_{n}a_{-n}}{[n]_y}\right)\right\>.\end{align*}
\end{remark}

\begin{remark}\label{resforwelschinger}
The refined Severi degrees $N^{\Delta,\delta}(y)$ and $N^{\Delta,\delta}(\alpha,\beta)(y)$ specialize to the Severi degrees $N^{\Delta,\delta}$ and $N^{\Delta,\delta}(\alpha,\beta)$ at $y=1$ and to the tropical Welschinger invariants 
$W^{\Delta,\delta}$ and $W^{\Delta,\delta}(\alpha,\beta)$ at $y=-1$. 
We write $[\ ,\ ]_{1}$, $\<\ | \ \>_{1}$, (resp. $[\ ,\ ]_{-1}$, $\<\ | \ \>_{-1}$) for the specializations of $[\ ,\ ]$ and $\<\ |\ \>$ at $y=1$ (resp. $y=-1$). 
Thus, to obtain the results for the Severi degrees  we just have to  specialize $y=1$ respectively $y=-1$  in \thmref{NDEL},
\corref{oneleft}, \corref{hexcor}, \corref{grading}. This is the same as replacing $[\ ,\ ]$ by $[\ ,\ ]_1$  respectively $[\ ,\ ]_{-1}$ in the commutation relation of the Heisenberg algebra, and $\<\ |\ \>$ by $\<\ |\ \>_1$  respectively $\< \ |\ \>_{-1}$

\begin{enumerate}
\item
As $[n]_1=n$, we get for $y=1$ the standard Heisenberg algebra modelled on the hyperbolic lattice
with commutation relations 
$[a_n,a_m]_1=0=[b_n,b_m]_1$, $[a_n,b_m]_1=n \delta_{n,-m}.$
and the inner product of the basis vectors $F$ is 
$\<v_{\mu,\nu}|v_{\mu',\nu'}\>_{1}=\prod_i \frac{i^{\mu_i}}{\mu_i!}\prod_j \frac{j^{\nu_j}}{\nu_j!}\delta_{\mu,\nu'}\delta_{\mu',\nu}.$

\item
For the specialization $y=-1$, we find that all $a_n,\ b_n$  with $n$ even lie in the center of the Heisenberg algebra. Therefore we will consider 
the Lie algebra $\H^{odd}$ generated by
the $a_n,\ b_n$,  with $n$ odd, with the commutation relations 
(for $n,m$ odd)
$$[a_n,a_m]_{-1}=0=[b_n,b_m]_{-1},\quad  [a_n,b_m]_{-1}=(-1)^{(n-1)/2} \delta_{n,-m}.$$
The Fock space $F^{odd}$ is  generated by applying the creation operators in $\H^{odd}$  to $v_\emptyset$. 
We call a partition $\mu=(1^{\mu_1},2^{\mu_2},\ldots)$ {\it odd} if $\mu_i=0$ for $i$ even.
Then a basis of $F^{odd}$ is given by the vectors $v_{\mu,\nu}=a_{-\mu}b_{-\nu} v_\emptyset$ with $\mu$ and $\nu$ odd, and the inner product of the basis vectors is given by
\begin{equation}\label{welin}\<v_{\mu,\nu}|v_{\mu',\nu'}\>_{-1}=\frac{(-1)^{(\|\mu\|+\|\nu\| -|\mu|-|\nu|)/2}}{\prod_{i}\mu_i!\prod_j \nu_j!}\delta_{\mu,\nu'}\delta_{\mu',\nu}.
\end{equation}
If we restrict attention to the {\it absolute} Welschinger invariants $W^{(S,L),\delta}$,  instead of the relative Welschinger invariants  $W^{(S,L),\delta}(\alpha,\beta)$,
we see that the right hand sides of the formulas of \thmref{NDEL}, \corref{oneleft}, \corref{hexcor}, are of the form
$\left\<v_{\mu,\nu}\left| Av_{\mu',\nu'}\right.\right\>$ for some element $A\in \H$, and $\mu,\nu$ satisfying  $\|\mu\|-|\mu|=\|\nu\|-|\nu|=0$. 
Therefore we can replace the inner product in \eqref{welin} by the standard inner product on $F^{odd}$
\begin{equation}\label{welinn}\<v_{\mu,\nu}|v_{\mu',\nu'}\>_{*}=\frac{\delta_{\mu,\nu'}\delta_{\mu',\nu}}{\prod_{i}\mu_i!\prod_j \nu_j!}.
\end{equation}

 For simplicity we only formulate the version for the Welschinger invariants of \corref{hexcor}:
Denote 
\begin{align*}
H_m^{odd}(t)&:=\sum_{k>0\ \text{odd}} b_{-k}b_{k}+\sum_{\|\mu\|=\|\nu\|-m }a_{-\mu}a_{-\nu}\\
G^{odd}_m(t)&:=\sum_{k>0\ \text{odd}} b_{-k}b_{k} +\sum_{\|\mu\|=\|\nu\|+1}
a_{-\mu}a_{\nu}+t\sum_{\|\mu\|=\|\nu\|-m}
a_{-\mu}a_{\nu}.
\end{align*}
where the second (respectively second and third) sums are now over pairs of odd partitions.
Then we have
\begin{align*}
W^{d,\delta}&=\left\<v_\emptyset\left|\Coeff_{t^d}\left[H_1^{odd}(t)^{d(d+3)/2-\delta}]\right]\right|v_{(1^d),\emptyset}\right\>_{*},\\
W^{(\Sigma_m,cF+dH),\delta}(y)&=
\left\< v_{(1^c),\emptyset}\left| \Coeff_{t^d}\left[H^{odd}_m(t)^{\binom{d+1}{2}m+cd+c+d-\delta}\right]\right|  v_{(1^{dm+c}),\emptyset}\right\>_{*},\\
W^{(\PP(1,(m-1),m),dH),\delta}(y)&=\left\< \Coeff_{t^d}\left[G^{odd}_{m-1}(t)^{\binom{m}{2}d^2+md-\delta}\right]\right\>_{*}.
\end{align*}
\end{enumerate}
\end{remark}

\section{Marked floor diagrams}
\label{sec:floor_diagrams}

Let $\Delta$ be an $h$-transverse polygon given by
$d_\Delta^b, d_\Delta^t\in \Z_{\ge 0}$ and multisets $r_\Delta,l_\Delta$ with
 $| l_\Delta|=|r_\Delta|$, such that $d_\Delta^t+\| r_\Delta\| =d_\Delta^b+\|l_\Delta\|$
as in Section~\ref{sec:h-transverse_polygons}. Recall that $h_\Delta =
|l_\Delta| = |r_\Delta|$ is the height of $\Delta$.

\begin{definition}
A \emph{$\Delta$-floor diagram} $\D$
consists of:
\begin{itemize}
\item
two orderings $R$ and $L$ of $r_{\Delta}$ and $l_{\Delta}$, and a
sequence $(s_1, \dots, s_{h_\Delta})$ of non-negative integers such
that  $|s| = d_\Delta^t$,
\item
a graph on a vertex set $\{1, \dots, h_\Delta \}$ of \emph{white vertices}, possibly with multiple edges, with edges directed $i \to j$ for $i<j$, and
\item
edge weights $w(e) \in \ZZ_{>0}$ for all edges $e$ such that
for every vertex $j$,
\begin{displaymath}
\dive(j) := \sum_{ \tiny
     \begin{array}{c}
  \text{edges }e\\
j \stackrel{e}{\to} k
     \end{array}
} w(e) -   \sum_{ \tiny
     \begin{array}{c}
  \text{edges }e\\
i \stackrel{e}{\to} j
     \end{array}
} w(e) \leq r_j - l_j + s_j. \,  \footnote{This inequality will become clear when we define marked floor diagrams.}
\end{displaymath}
\end{itemize}
We call $\aa:=(d^{\,t},R-L)$ the \emph{divergence sequence}.
\end{definition}

\begin{figure}[!htb]
\begin{center}
\begin{picture}(50,23)(30,-10)\setlength{\unitlength}{4pt}\thicklines
\oooo\Eeee\eEee\eeEe
\qbezier(10.8,-0.6)(13.75,-3)(20.5,-3)\qbezier(20.5,-3)(27,-3)(29.2,-0.6)
\put(25,1.5){\makebox(0,0){$2$}} 
\end{picture}
\end{center}
\caption{Floor diagram with $R = (1, 1, 1, 1)$, $L = (0, 0, 0, 0)$,
  and $S = (0,0,0,0)$.}
\label{fig:floor_diagram}
\end{figure}
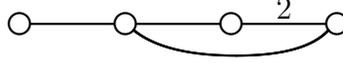

\begin{definition}\label{def:marking}
A \emph{marked floor diagram} or \emph{marking} is obtained from a floor diagram $\D$ as follows:.

{\bf Step 1:} For each vertex $j$ of $\D$, create $s_j$ new
indistinguishable \emph{black vertice}s and
connect them to $j$ with new edges directed towards $j$. 

{\bf Step 2:} For each vertex $j$ of $\D$, create $R_j - L_j + s_j- \textrm{div}(j)$ new indistinguishable vertices and
connect them to $j$ with new edges directed away from $j$. This makes the divergence of vertex $j$ equal to $R_j - L_j$ for $1 \leq j \leq h_\Delta$.

{\bf Step 3:} Subdivide each edge of the original floor
diagram $\D$ into two
directed edges by introducing a new
vertex for each edge. The new edges inherit their weights and orientations. Denote the
resulting graph $\tilde{\D}$.

{\bf Step 4:} Linearly order the vertices of $\tilde{\D}$
extending the order of the vertices of the original floor
diagram $\D$ such that, as before, each edge is directed from a
smaller vertex to a larger vertex.

The extended graph $\Gamma$
together with the linear order on
its vertices is called a \emph{marked floor diagram}

\begin{figure}[!htb]
\begin{center}
\begin{picture}(50,35)(145,-8)\setlength{\unitlength}{4pt}\thicklines
\put(0,0){\circle{2}}
\put(8,0){\circle*{2}}
\put(16,0){\circle{2}}
\put(24,0){\circle*{2}}
\put(32,0){\circle*{2}}
\put(40,0){\circle{2}}
\put(48,0){\circle*{2}}
\put(56,0){\circle{2}}
\put(64,0){\circle*{2}}
\put(72,0){\circle*{2}}
\put(80,0){\circle*{2}}
\put(88,0){\circle*{2}}

\put(1,0){\line(1,0){6}}
\put(9,0){\line(1,0){6}}
\put(17,0){\line(1,0){6}}
\put(33,0){\line(1,0){6}}
\put(41,0){\line(1,0){6}}
\put(49,0){\line(1,0){6}}
\put(57,0){\line(1,0){6}}

\put(44,2){\makebox(0,0){$2$}}
\put(52,2){\makebox(0,0){$2$}}

\qbezier(16.8,0.6)(19,3)(24,3)\qbezier(24,3)(29,3)(31.2,0.6)
\qbezier(24.8,-0.6)(30,-5)(40,-5)\qbezier(40,-5)(49,-5)(55.2,-0.6)
\qbezier(56.8,-0.6)(59,-3)(64,-3)\qbezier(64,-3)(69,-3)(71.2,-0.6)
\qbezier(56.8,-0.6)(61,-5)(68,-5)\qbezier(68,-5)(74,-5)(79.2,-0.6)
\qbezier(56.8,0.6)(62,5)(72,5)\qbezier(72,5)(81,5)(87.2,0.6)
\end{picture}
\end{center}
\caption{Marking for floor diagram of Figure~\ref{fig:floor_diagram}.}
\end{figure}
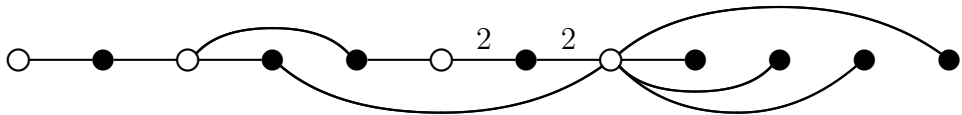
\end{definition}

We need to count marked floor diagrams up to equivalence. Two such $\Gamma_1$, $\Gamma_2$ are \emph{equivalent} if
$\Gamma_1$ can be obtained from $\Gamma_2$ by permuting edges
without changing their weights; \emph{i.e.,} if there exists an
automorphism of weighted graphs which preserves the vertices of $\D$ and maps $\Gamma_1$ to
$\Gamma_2$.

If $\Gamma$ is a marked floor diagram obtained from a $\Delta$-floor
diagram we label its $k$~vertices $\{1,\ldots,k\}$ left to
right.
The {\em cogenus} of $\Gamma$ is $\delta(\Gamma):=\#\Delta-1-k$.

Our definition of the cogenus of a marked floor diagrams agrees
with \cite{FM, AB10}: let $\Gamma$ be a marked $\Delta$-floor
diagram corresponding to a tropical curve $C$ of degree $\Delta$ with
cogenus $\delta(C)$ through a (vertically stretched, see
\cite[Definition~3.6]{BG14}) tropical point configuration $\Pi$. Then
$|\Pi| = \# \Delta -1 - \delta(C)$. The vertices of $\Gamma$ correspond
to the points in $\Pi$, see \cite[Section~5.2]{BM2}, so $\Gamma$ has
$\# \Delta -1 - \delta(C)$ vertices. We defined the cogenus
$\delta(\Gamma)$ precisely so that $\delta(\Gamma)= \delta(C)$.

The \emph{refined multiplicity} of $\Gamma$ is 
$$\mult(\Gamma,y):=\prod_{\text {edges $e$ of $\Gamma$}} [w(e)]_y.$$
We can compute the refined Severi degrees $N^{\Delta, \delta}$, for
$h$-transverse $\Delta$, in terms of marked floor diagrams.

\begin{theorem}[{\cite[Theorem~2.7]{AB10}}]
For any $h$-transverse polygon $\Delta$ and any $\delta \ge 0$:
\[
N^{\Delta, \delta} = \sum_{[\Gamma]} \mult(\Gamma, y).
\]
The sum is over equivalence classes $[\Gamma]$ of marked $\Delta$-floor diagrams of cogenus $\delta$.
\end{theorem}
In \cite[Theorem~2.7]{AB10} the formula is instead a sum of $\mult(\D, y)\nu(\D)$ over all floor diagrams $\D$ of cogenus $\delta$, where $\mult(\D)$ is the number of markings of $\D$ up to isomorphism, but this is clearly equivalent.
 Note that if $\Gamma$ is a  marking of $\D$, then 
  $\mult(\Gamma, y)=\mult(\D,y)$, because all the inner edges of $\D$ (which are the only ones with multiplicity different from $1$) are divided in $\Gamma$ into 
  two edges.

\begin{remark}
We find the following explicit description of marked $\Delta$-floor diagrams.
A marked $\Delta$-floor diagram $\Gamma$ of cogenus $\delta$ is a directed colored graph with vertex set $\{1,\ldots,\#\Delta-1-\delta\}$.
All edges are directed $i\to j$ with $i<j$. The vertices $i$ which lie only on edges $i\to j$ with $i<j$ are called \emph{source vertices}, and those which lie only on edges 
$j\to i$ are called \emph{sink vertices}. 
The diagrams are required to satisfy the following:
\begin{enumerate}
\item There are $h_\Delta$ white vertices $\{i_1,\ldots,i_{h_{\Delta}}\}$, the other vertices are black, all edges connect vertices of different colors.
\item A black vertex that is not a source or sink vertex has precisely one incoming and one outgoing edge, both of the same weight.
\item There are precisely $d^t_\Delta$ black source vertices and $d^b_\Delta$ black  sink vertices.
\item All the edges connected to black source or sink vertices have weight $1$.
\item There are orderings $R=(R_1,\ldots, R_{h_\Delta})$ and $L=(L_1,\ldots, L_{h_\Delta})$ of $r_\Delta$ and $l_\Delta$, such that for every white vertex $i_j$
we have \begin{displaymath}
\dive(i_j) := \sum_{ \tiny
     \begin{array}{c}
  \text{edges }e\\
i_j \stackrel{e}{\to} k
     \end{array}
} w(e) -   \sum_{ \tiny
     \begin{array}{c}
  \text{edges }e\\
i \stackrel{e}{\to} i_j
     \end{array}
} w(e)= R_j - L_j.
\end{displaymath}
\end{enumerate}
The equivalence relation is by permuting the edges of the same weight, leaving all the white vertices fixed.
\end{remark}

As in the case of tropical curves earlier, there is a relative version
of marked floor diagrams. Let $\alpha$ and $\beta$ be two sequences with
$\|\alpha\| + \|\beta\| = d_\Delta^b$. 

\begin{definition}\label{def:albemarking}
An \emph{$(\alpha, \beta)$-marked floor diagram} or \emph{$(\alpha, \beta)$-marking} of a floor diagram $\D$ is defined as follows:

{\bf Step 1:} As Step~1 in Definition~\ref{def:marking}.

{\bf Step 2:} Fix a pair of collections of sequences $( \{ \alpha^i
\}, \{ \beta^i \} )$, where $i$ runs over the vertices of $\D$, with:
\begin{enumerate}
\item The sums over each collection satisfy $\sum_i \alpha^i =
  \alpha$ and  $\sum_i \beta^i =
  \beta$.
\item For all vertices $i$ of $\D$ we have $\sum_{j \ge 1} j(\alpha_j^i+ \beta_j^i)= R_i - L_i + s_i- \textrm{div}(i)$.
\end{enumerate} 
The second condition says that the ``degree of the pair $(\alpha^i,
\beta^i)$'' is compatible with the divergence at vertex $i$.
Each such pair $( \{ \alpha^i
\}, \{ \beta^i \} )$ is called \emph{compatible} with $\D$ and $(\alpha,
\beta)$.

{\bf Step 3:}  For each vertex $i$ of $\D$ and every $j \ge 1$ create
$\beta_j^i$ new
vertices, called \emph{$\beta$-vertices} and illustrated as
\begin{picture}(4,0)(0,0)\setlength{\unitlength}{4pt}\thicklines
\put(1,1){\circle*{2}}
\end{picture}
, and connect them to $i$ with new edges of weight $j$
directed away from $i$.  For each vertex $i$ of $\D$ and every $j \ge 1$ create
$\alpha^i_j $ new
vertices, called \emph{$\alpha$-vertices} and illustrated as
\begin{picture}(4,0)(0,0)\setlength{\unitlength}{4pt}\thicklines
\put(1,1){\circle{2}}
\put(1,1){\circle*{1}}
\end{picture}
, and connect them to $i$ with new edges of weight $j$
directed away from $i$.

{\bf Step 4:} Subdivide each edge of the original floor
diagram $\D$ into two
directed edges by introducing a new
vertex for each edge. The new edges inherit their weights and orientations. Denote the
resulting graph $\tilde{\D}$.

{\bf Step 5:} Linearly order the vertices of $\tilde{\D}$
extending the order of the vertices of the original floor
diagram $\D$ such that, as before, each edge is directed from a
smaller vertex to a larger vertex. Furthermore, we require that the $\alpha$-vertices are largest among
all vertices, and
for every pair of $\alpha$-vertices $i' > i$, the weight of the $i'$-adjacent
edge is larger than or equal to the weight of the $i$-adjacent edge.
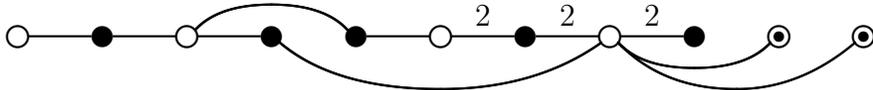
\begin{figure}[!htb]
\begin{center}
\begin{picture}(50,35)(145,-8)\setlength{\unitlength}{4pt}\thicklines
\put(0,0){\circle{2}}
\put(8,0){\circle*{2}}
\put(16,0){\circle{2}}
\put(24,0){\circle*{2}}
\put(32,0){\circle*{2}}
\put(40,0){\circle{2}}
\put(48,0){\circle*{2}}
\put(56,0){\circle{2}}
\put(64,0){\circle*{2}}
\put(72,0){\circle{2}}
\put(72,0){\circle*{1}}
\put(80,0){\circle{2}}
\put(80,0){\circle*{1}}

\put(1,0){\line(1,0){6}}
\put(9,0){\line(1,0){6}}
\put(17,0){\line(1,0){6}}
\put(33,0){\line(1,0){6}}
\put(41,0){\line(1,0){6}}
\put(49,0){\line(1,0){6}}
\put(57,0){\line(1,0){6}}

\put(44,2){\makebox(0,0){$2$}}
\put(52,2){\makebox(0,0){$2$}}
\put(60,2){\makebox(0,0){$2$}}

\qbezier(16.8,0.6)(19,3)(24,3)\qbezier(24,3)(29,3)(31.2,0.6)
\qbezier(24.8,-0.6)(30,-5)(40,-5)\qbezier(40,-5)(49,-5)(55.2,-0.6)
\qbezier(56.8,-0.6)(59,-3)(64,-3)\qbezier(64,-3)(69,-3)(71.2,-0.6)
\qbezier(56.8,-0.6)(61,-5)(68,-5)\qbezier(68,-5)(74,-5)(79.2,-0.6)
\end{picture}
\end{center}
\caption{$(1^2, 2^1)$-marking for floor diagram of
  Figure~\ref{fig:floor_diagram}.}
\label{fig:rel_floor_diagram_marking}
\end{figure}
\end{definition}

The extended graph $\Gamma$
 together with the linear order on
its vertices is called an
\emph{$(\alpha, \beta)$-marking} of the original floor diagram $\D$,
  and the diagram obtained this way is called an
\emph{$(\alpha,\beta)$-marked $\Delta$-floor diagram}. 
We count $(\alpha,\beta)$-marked floor
diagrams up to equivalence. Two $(\alpha,\beta)$-marked $\Delta$-floor diagrams
$\Gamma_1$, $\Gamma_2$  are \emph{equivalent} if there exists a weight preserving isomorphism of weighted
graphs mapping $\Gamma_1$ to $\Gamma_2$ which fixes the white vertices.
The \emph{refined multiplicy} of $\Gamma$ is
\[
\mult(\Gamma, y) = \prod_e [w(e)]_y
\]
where the product ranges over all edges $e$ of $\Gamma$ excluding edges adjacent
to $\alpha$-vertices.

The following theorem can be proved by combining the argument of the
proofs of \cite[Proposition 7.7]{BG14} ($S = \PP^2$ with tangency
conditions) and \cite[Theorem 5.7]{BG14} ($S = S(\Delta)$ without tangency
conditions) in a straightforward way.

\begin{theorem}
For any $h$-transverse polygon $\Delta$, any $\delta \ge 0$, and any
pair of sequences $\alpha$ and $\beta$ with $\|\alpha\| + \|\beta\| = d_\Delta^b$:
\[
N^{\Delta, \delta}(\alpha, \beta)(y) = \sum_{[\Gamma]} \mult(\Gamma, y) ,
\]
with the sum over the equivalence classes of $(\alpha,\beta)$-marked $\Delta$-floor diagrams.
\end{theorem}

For the comparison with Feynman graphs, we want to consider a slightly modified version of the $(\alpha,\beta)$-marked $\Delta$-floor diagrams, which we call
{\it extended} $(\alpha,\beta)$-marked $\Delta$-floor diagrams:

{\bf Step 6:}
 Given an $(\alpha,\beta)$-marked $\Delta$-floor diagram $\Gamma$, with vertices $\{1,\ldots,k\}$ we define its extension $\widehat \Gamma$, 
by adding a white vertex $0$ with $d^t_\Delta$ outgoing  edges of weight 1 and connecting it to the $d^t_\Delta$ black source vertices of $\Gamma$.
Furthermore we add a white vertex $k+1$ and connect it to the $|\beta|$ black sink vertices $l_i$ of $\Gamma$, each by an edge of the same weight as the edge ending in $l_i$.
Finally we make the $\alpha$-vertices black.

\begin{figure}[!htb]
\begin{center}
\begin{picture}(50,55)(145,-25)\setlength{\unitlength}{4pt}\thicklines
\put(0,0){\circle{2}}
\put(8,0){\circle*{2}}
\put(16,0){\circle{2}}
\put(24,0){\circle*{2}}
\put(32,0){\circle*{2}}
\put(40,0){\circle{2}}
\put(48,0){\circle*{2}}
\put(56,0){\circle{2}}
\put(64,0){\circle*{2}}
\put(72,0){\circle*{2}}
\put(80,0){\circle*{2}}
\put(88,0){\circle{2}}

\put(1,0){\line(1,0){6}}
\put(9,0){\line(1,0){6}}
\put(17,0){\line(1,0){6}}
\put(33,0){\line(1,0){6}}
\put(41,0){\line(1,0){6}}
\put(49,0){\line(1,0){6}}
\put(57,0){\line(1,0){6}}

\put(44,2){\makebox(0,0){$2$}}
\put(52,2){\makebox(0,0){$2$}}
\put(60,2){\makebox(0,0){$2$}}
\put(76,7){\makebox(0,0){$2$}}

\qbezier(16.8,0.6)(19,3)(24,3)\qbezier(24,3)(29,3)(31.2,0.6)
\qbezier(24.8,-0.6)(30,-5)(40,-5)\qbezier(40,-5)(49,-5)(55.2,-0.6)
\qbezier(56.8,-0.6)(59,-3)(64,-3)\qbezier(64,-3)(69,-3)(71.2,-0.6)
\qbezier(56.8,-0.6)(61,-5)(68,-5)\qbezier(68,-5)(74,-5)(79.2,-0.6)
\qbezier(64.8,0.6)(69,5)(76,5)\qbezier(76,5)(82,5)(87.2,0.6)

\end{picture}
\end{center}
\caption{Extension of Figure~\ref{fig:rel_floor_diagram_marking}. Note
  that this diagram is also a Feynman diagram for $a_{(1)} b_{-1} b_{1} a_{-(1)}
  a_{(1^2)} (b_{-1} b_{1})^2 a_{-(1)} a_{(2)} b_{-2} b_{2} a_{-(1,2)}
  a_{(1^2, 2)} b_{-2} b_{2} (b_{-1})^2 a_{-(2)}$ (see below).}
\end{figure}

\begin{remark}\label{extendmark}
We find the following explicit description of extended $(\alpha,\beta)$-marked $\Delta$-floor diagrams.
An $(\alpha,\beta)$-marked $\Delta$-floor diagram of cogenus $\delta$
is a
directed colored weighted graph $\Gamma$  with vertex set $\{0,\ldots,t+1\}$ with $t=\#\Delta-1-\delta-\|\alpha\|-\|\beta\|+|\alpha|+|\beta|$, we also let  $s:=t-|\alpha|$ and for each $l\ge 0$ let $s_l=s+\sum_{j\le l}\alpha_j$.
All edges are directed $i\to j$ with $i<j$. 
The diagrams are required to satisfy the following:
\begin{enumerate}
\item There are indices $0<i_1<\ldots<i_{h_{\Delta}}\le s$, such that the white vertices are $\{0,i_1,\ldots,i_{h_{\Delta}},t+1\}$. The other vertices are black. $0$, $s+1, \ldots, t+1$,  are sink vertices.
\item All edges connect vertices of different colors. 
\item Each black vertex, with the exception of  the sink vertices,  has precisely one incoming and one outgoing edge of the same weight. 
\item The edges connected to $0$, $s+1, \ldots, t+1$ are the following:
\begin{enumerate}
\item
$0$ has  $d^t_{\Delta}$ outgoing edges, all of weight $1$. 
\item For $i=s_{k}+1,\ldots,s_{k+1}$ the vertex $i$ has one incoming  edge of weight $k$.
\item $t+1$ has  $\beta_i$ incoming edges of weight $i$ for all $i$. 
\end{enumerate}
\item There are orderings $R=(R_1,\ldots, R_{h_\Delta})$ and $L=(L_1,\ldots, L_{h_\Delta})$ of $r_\Delta$ and $l_\Delta$, such that for every white vertex $i_j\not \in\{0,t+2\}$
we have \begin{displaymath}
\dive(i_j) := \sum_{ \tiny
     \begin{array}{c}
  \text{edges }e\\
i_j \stackrel{e}{\to} k
     \end{array}
} w(e) -   \sum_{ \tiny
     \begin{array}{c}
  \text{edges }e\\
i \stackrel{e}{\to} i_j
     \end{array}
} w(e)= R_j - L_j.
\end{displaymath}
\end{enumerate}
An equivalence $\widehat\Gamma_1\to \widehat\Gamma_2$ of  extended $(\alpha,\beta)$-marked $\Delta$-floor diagrams is an isomorphism of weighted directed graphs by permutation of the edges, fixing the white vertices. We also put 
$$\mult(\widehat \Gamma,y):=\prod_{e \text{ edges of }\widehat \Gamma} [w(e)]_y.$$
\end{remark}

\begin{corollary}
$$N^{\Delta,\delta}(\alpha,\beta)=\frac{1}{I_y^{\alpha+\beta}} \sum_{[\widehat \Gamma]} \mult(\widehat\Gamma,y).$$
The sum is over equivalence classes of extended $(\alpha,\beta)$-marked $\Delta$-floor diagrams. 
\end{corollary}
\begin{proof} The result follows from the following easy facts:
(1) $(\alpha,\beta)$-marked floor diagram and extended $(\alpha,\beta)$-marked floor diagram are obviously in bijection, and the same holds for the equivalence classes. (2)  By definition, if $\Gamma$ is an $(\alpha,\beta)$-marked floor diagram, and $\widehat \Gamma$ is its extension, then by definition
$$\mult(\widehat
\Gamma,y)=\mult(\Gamma,y)\prod_{i}([i]_y^{\alpha_i+\beta_i}).$$
\end{proof}

\section{Feynman diagrams}\label{Feynman}

We associate Feynman diagrams to monomials $M=m_0 \cdots m_l$, where each $m_i$ is of the form $a_{-\mu}a_{\nu}$ or  $b_{-k}b_{k}$ or $b_{-k}$ for partitions $\mu,\nu$
and positive integers $k$.

\begin{definition}\label{Feyndia}
A \emph{Feynman diagram}  for $M= m_0\cdots m_l$ as above is a directed weighted graph with vertices $0,\ldots,l$.
If 
$m_i=a_{-\mu}a_{\nu}$  then the vertex $i$ is white with $|\mu|$
incoming edges with $\mu_j$ of weight $j$ for all $j$ and $|\nu|$
outgoing  edges with $\nu_j$ of weight $j$ for all $j$. If $m_i=b_{-k}b_k$, then the vertex $i$ is black with one incoming and one outcoming edge of both weight $k$. 
If $m_i=b_{-k}$, then the vertex $i$ is black with one incoming edge of  weight $k$. 
 All edges are directed $i\to j$ with $i<j$ and they connect vertices of different color.
 An equivalence $\Gamma_1\to \Gamma_2$ of Feynman diagrams for $M$ is an isomorphism of weighted directed colored graphs, by permutation of the edges of the same weight, leaving the white vertices fixed.
For an edge $e$ of a Feynman diagram we denote its weight by $w(e)$. 
For a Feynman diagram $\Gamma$, its multiplicity is 
$$\mult(\Gamma,y):=\prod_{e \text{ edges of }\Gamma} [w(e)]_y.$$
\end{definition}
The following can be viewed as a version 
of the classical Wick's theorem \cite{W}.
\begin{proposition}[Wick's theorem] Let $M=m_1\cdots m_l$ be a monomial in the $a_{-\mu} a_{\nu}$, $b_{-k}b_{k}$, $b_{-k}$. Then
$$\<M\>=\sum_{[\Gamma]} \mult(\Gamma,y),$$
where the sum runs over all equivalence classes of Feynman diagrams for $M$.
\end{proposition}

\begin{proof}
To simplify notations we will write $a[0]_i:=a_i$, $a[1]_i:=b_i$. Note that the commutation relations \eqref{Heiscomm} are 
$$\big[a[s]_i,a[t]_j\big]=[i]_y\delta_{i,-j}\delta_{s,1-t},\quad s,t\in\{0,1\}.$$

(1) Now let $N=a[s_1]_{i_1}\cdots a[s_l]_{i_l}$ be any monomial in the $a[s]_{i}$, $i\in \Z_{\ne0}$, $s\in \{0,1\}$.
We compute $\<N\>$. If $i_1\le 0$,  then 
$$\<N\>=\big\<a[s_1]_{-i_1}v_{\emptyset}\bigm|a[s_2]_{i_2}\cdots a[s_l]_{i_l}v_{\emptyset}\big\>=0,$$
because $a[s_1]_{-i_1}v_{\emptyset}=0$.
If $i_1>0$, applying the commutation relation, we get 
$$\<N\>=\sum_{\big\{m \bigm |\ i_{m}=-i_{1},\ s_{m}=1-s_{1}\big\}} [i_1]_{y}\big\<a[s_2]_{i_2}\cdots\widehat{a[s_m]_{i_m}}\cdots a[s_l]_{i_l}\big\>,$$ where the
$\widehat{\phantom{m}}$ means that the factor is removed.
By induction this gives that
$\<N\>=0$, if $l$ is odd, and if $l=2m$, then
\begin{equation}\label{pairs}\<N\>=\sum_{I_1,\ldots,I_m} \prod_{j=1}^m [w(I_j)]_y.\end{equation}
Here the sum is over all decompositions of $\{1,\ldots, 2m\}$ into disjoint subsets of $2$ elements
$I_j=\{I_j^1,I_j^2\}$, with the following properties:
$$I_j^1<I_j^2,\quad i_{I_j^1}>0,\quad  i_{I_j^2}=-i_{I_j^1}, \quad s_{I_j^2}=1-s_{I_j^1},$$ 
and we write 
$w(I_{j})=i_{I_j^1}$. 

We can view this as a count of directed graphs with multiplicities.   For each factor $a[s_m]_{i_m}$ we place a vertex at the point $m$.
The vertex is white if $s_m=0$ and black if $s_m=1$. It has one incoming edge of weight $-i_m$ if $i_m<0$, it has one outgoing edge of weight $i_m$ if $i_m>0$, and no further edges. A graph $\Gamma$ for $\<N\>$ is a directed graph  with these vertices and edges, so that  every edge contains two vertices, and every vertex is connected by an edge to one other of a different color. We denote $w(e)$ the weight of the edges of $\Gamma$. The multiplicity of $\Gamma$ is 
$$\mult(\Gamma,y)=\prod_{e \text{ edge of }\Gamma} [w(e)]_y.$$
 Clearly there is a bijection 
between the decomposion of $1,\ldots,l$ into disjoint pairs of integers as in \eqref{pairs} and the graphs for $\<N\>$.
Thus by \eqref{pairs} we have 
$$\<N\>=\sum_{\Gamma\text{ graph for } N} \mult(\Gamma,y).$$

(2) Now let again $M=m_1\cdots m_l$ be a monomial in the $(a_{-\mu}a_{\nu})$,  $(b_{-k}b_{k})$, $b_{-k}$. 
Assume that the factors of the form $(a_{-\mu}a_{\nu})$ are $m_{i_1},\ldots, n_{i_n}.$ Then write 
$m_{i_s}=a_{-\mu^s}a_{\nu^s}$ for $s=1,\ldots,n$, for partitions $\mu^s=(1^{\mu^s_1},2^{\mu^s_2},\ldots)$, $\nu^s=(1^{\nu^s_1},2^{\nu^s_2},\ldots)$.
Then $M=\frac{1}{\prod_{s=1}^n \mu^s!\nu^s!} N$,  where 
$N=\prod_{i=1}^l \tilde m_i,$ with $\tilde m_i=m_i$ for $i\not \in \{i_1,\ldots,i_n\}$ and $\tilde m_{i_{s}}=\left(\prod_{j} a_{-j}^{\mu^s_j}\right)\left(\prod_{j} a_{j}^{\nu^s_j}\right).$
Thus $$\<M\>=\frac{1}{\prod_{s=1}^n \mu^s!\nu^s!}\sum_{\Gamma} \mult(\Gamma,y),$$
 where $\Gamma$ runs through the diagrams for $\<N\>$.
 
  Finally we need to relate the Feynman diagrams for $N$ to those for $M$.
 We obtain all the Feynman diagrams for $M$ from the diagrams for $N$ by replacing all the vertices corresponding to one factor
$(a_{-\mu^s}a_{\nu^s})$ by one white vertex, and replacing the two vertices corresponding to a factor $b_{-k}b_k$ by one black vertex, and keeping all the edges and their multiplicities and orientations, and have an edge connected to a vertex in $M$
if it was connected to one of the vertices in $N$ corresponding to it. Under this operation the $\prod_{s=1}^n \mu^s!\nu^s!$ graphs corresponding to the reorderings of the factors in each of the factors $\prod_{j}(a_{-j})^{\mu^s_j}\cdot \prod_j (a_j)^{\nu^s_j}$ are mapped to equivalent Feynman diagrams for $M$
(and two diagrams are mapped to equivalent Feynman diagrams if and only if they are related by such a reordering). The multiplicities of the graphs are preserved.
Thus the  claim follows.
\end{proof}

\begin{proof}[Proof of \thmref{NDEL}]
It is enough to prove \eqref{NDEL2}, because \eqref{NDEL1} is a special case.

Let $t:=\#\Delta-\delta-1-\|\alpha\|-\|\beta\|+|\alpha|+|\beta|$ and $s=t-|\alpha|$.
We can write $I_y^{\alpha+\beta}$ times the right hand side of \eqref{NDEL2} as
\begin{equation}\label{Feyneq}
\alpha!\left\<a_{(1^{d_\delta^t})}\left(\sum_{R,L} \Coeff_{T^{R-L}}[H(T)^s]\right)b_{-\alpha}a_{-\beta}\right\>=
\left\<a_{(1^{d_\delta^t})}\left(\sum_{R,L} \Coeff_{T^{R-L}}[H(T)^s]\right)\left(\prod_{i}b_{-i}^{\alpha_i}\right)a_{-\beta}\right\>,\end{equation}
with $R$, $L$ running through the orderings of $r_\Delta$ and $l_{\Delta}$.
Thus in order to prove the theorem it is enough to show:

{\bf Claim:} {\it The Feynman diagrams for the right hand side of \eqref{Feyneq} are precisely the extended $(\alpha,\beta)$-marked $\Delta$-floor diagrams.}

Fix $R=(R_1,\ldots,R_{h_{\Delta}})$, $L=(L_1,\ldots,L_{h_{\Delta}})$.
Then by definition 
$$a_{(1^{d_\Delta^t})}\Coeff_{T^{R-L}}[H(T)^s]\left(\prod_{i}b_{-i}^{\alpha_i}\right)a_{-\beta}$$ 
is the sum over all monomials $m_0\cdots m_{t+1}$, satisfying the following:
\begin{enumerate}
\item $m_0=a_{(1^{d_\Delta^t})}$
\item There exist indices $0< i_1<\ldots<i_{h_{\Delta}}\le s$ such that 
$m_{i_j}= a_{-\mu^j}a_{\nu^j}$ for partitions $\mu^j,\nu^j$  with  $\|\nu_j\|-|\mu_j\|=R_j-L_j$ for $j=1,\ldots, h_\Delta$.
\item The other $m_l$ with $l\le s$ are of the form
$b_{-k_l}b_{k_l}$ for some $k_l$.
\item For all $l\ge 0$ let $s_l=s+\sum_{i\le l} \alpha_i.$ Then $m_i=b_{-i}$ for all $i\in \{s_{l}+1,\ldots,s_{l+1}\}$.
\item $m_{t+1}=a_{-\beta}$.
\end{enumerate}
Now by \defref{Feyndia} the Feynman diagrams up to equivalence for these monomials are precisely the extended $(\alpha,\beta)$-marked $\Delta$-floor diagrams
up to equivalence, as described in \remref{extendmark}.
\end{proof}

\bibliographystyle{amsplain}
\bibliography{References_Florian}

\end{document}